\documentclass[11pt]{amsart}                  
\usepackage{amssymb,amsmath,amsfonts,latexsym}
\usepackage[margin=1.1in]{geometry}
\usepackage{enumerate}
\usepackage{slashed}
\usepackage{xypic}

\newcommand{\complex}{{\mathbb C}}

\newcommand{\C}{{\mathbb{C}}}
\newcommand{\R}{{\mathbb{R}}}

\newcommand{\N}{{\mathbb{N}}}

\newcommand{\cale}{{\mathcal E}}
\newcommand{\calf}{{\mathcal F}}

\newcommand{\calh}{{\mathcal H}}

\newcommand{\call}{{\mathcal L}}
\newcommand{\calm}{{\mathcal M}}

\newcommand{\cals}{{\mathcal S}}
\newcommand{\calt}{{\mathcal T}}

\newcommand{\calu}{{\mathcal U}}

\newcommand{\sfG}{\mathsf{G}}

\newcommand{\cea}{C_e(A^*)}
\newcommand{\cEg}{C^*_{\mathcal{E}}(\sfG)}
\newcommand{\cEa}{C^*_{\mathcal{E}}(A)}

 \DeclareMathOperator{\Hom}{Hom}

\DeclareMathOperator{\Ind}{Ind}

\newtheorem{thm}{Theorem}[section]
\newtheorem{prop}[thm]{Proposition}
\newtheorem{cor}[thm]{Corollary}
\newtheorem{lem}[thm]{Lemma}
\theoremstyle{definition}
\newtheorem{rem}[thm]{Remark}
\newtheorem{ex}[thm]{Example}
\theoremstyle{definition}
\newtheorem{defn}[thm]{Definition}
\newtheorem{assump}[thm]{Assumption}

\newcommand{\be}{\begin{eqnarray}}
\newcommand{\ee}{\end{eqnarray}}

\theoremstyle{plain}

        \usepackage{hyperref}
        
\begin{document}
\title[Roe $C^*$-algebra and vanishing theorem]{Roe $C^*$-algebra for groupoids and generalized Lichnerowicz Vanishing theorem for foliated manifolds }
\author{Xiang Tang, Rufus Willett, and Yi-Jun Yao}
\maketitle

\begin{abstract}We introduce the concept of Roe $C^*$-algebra for a locally compact groupoid whose unit space is in general not compact, and that is equipped with an appropriate coarse structure and Haar system. Using Connes' tangent groupoid method, we introduce an analytic index for an elliptic differential operator on a Lie groupoid equipped with additional metric structure, which takes values in the K-theory of the Roe $C^*$-algebra. We apply our theory to derive a Lichnerowicz type vanishing result for foliations on open manifolds. 
\end{abstract}

\section{Introduction}

In this article, we study the index theory of longitudinal elliptic operators on open foliated manifolds.   Our study is inspired by the longitudinal index theory on a manifold $M$ with a regular foliation $\calf$ due to Connes and Connes-Skandalis (see for example \cite{co79}, \cite{co86}, and \cite{co-sk}).  Examples of longitudinal elliptic operators on open foliated manifolds appear naturally, both in their own right and also by associating non-compact manifolds to closed manifolds as in the next two examples.  

As a first example, in order to define the transverse fundamental class on $(M, \mathcal{F})$ in \cite{co86}, Connes introduced a noncompact manifold $\mathcal{M}$, the space of metrics on the normal bundle of $\calf$, together with a regular foliation $\widetilde{\calf}$ of $\mathcal{M}$. A longitudinal elliptic operator on $(M, \calf)$ has a canonical lifting to a longitudinal elliptic operator on $(\mathcal{M}, \widetilde{\calf})$.  The lifted operator on $(\mathcal{M}, \widetilde{\calf})$ plays a key role in Connes' theory of transverse fundamental classes.

A second source of examples of longitudinal elliptic operators on open manifolds is our efforts to incorporate the fundamental group into index theory on a foliated manifold $(M, \calf)$.  Here, we consider the covering space $\widehat{M}$ of $M$ together with the lifted regular foliation $\widehat{\calf}$. A longitudinal elliptic operator on $(M, \calf)$ naturally lifts to a longitudinal elliptic operator on $(\widehat{M}, \widehat{\calf})$.\\

To allow the possibility that the rank of $\calf$ can vary on $M$, we develop our index theory for elliptic differential operators on a general Lie groupoid $\mathsf{G}$ with unit space $\sfG_0$.  Our theory applies to a manifold $M$ with a regular foliation $\calf$ when the groupoid $\mathsf{G}$ is the holonomy groupoid associated to a foliation $\calf$ on $M=\sfG_0$.

If $\sfG_0$ is a closed manifold, the index of a longitudinally elliptic differential operator on $\sfG$ defines an element in the K-theory of the (reduced) groupoid $C^*$-algebra $C^*_r(\sfG)$ of the groupoid $\sfG$.  However, this is generally not the case when $\sfG_0$ is not closed.  In general, we show how to use a Euclidean structure on the Lie algebroid $A$ to define a coarse structure $\cale$ on $\sfG$ in the sense of \cite{hig-ped-roe:controlled} under a natural completeness assumption (see Proposition \ref{prop:groupoid-entourage} for details).  Following ideas that start in \cite{ro:open}, we then introduce a Roe $C^*$-algebra $\cEg$ for $\sfG$ associated to the coarse structure $\mathcal{E}$; its $K$-theory provides a natural home for indices of longitudinal elliptic operators in the non-compact case.  If $\sfG_0$ happens to be closed, $\cEg$ is just $C^*_r(\sfG)$, while if $\sfG$ is the pair groupoid on a Riemannian manifold $M$, $\cEg$ is the classical Roe algebra $C^*(M)$ of $M$ \cite[Chapter 3]{Roe:index}; thus our theory generalizes both of these cases.  

Assume now that the Lie algebroid $A$ associated to $\sfG$ is equipped with a Euclidean structure, and $\sfG$ with the coarse structure it induces.  We also associate an algebra $\cea$ of continuous functions on $A$ whose $K$-theory provides a natural home for symbols of geometrically defined operators.  A coarse version of the groupoid $C^*$-algebra of Connes' tangent groupoid $\calt\sfG$ of $\sfG$ (c.f. \cite{connes:book}) leads to the following coarse analytic index map
\[
\Ind_a: K_\bullet\big(\cea\big)\longrightarrow K_\bullet\big(\cEg\big).
 \]
This approach to `coarse index theory' via tangent groupoids seems to be new even in the classical case when $\sfG$ is the pair groupoid associated to a Riemannian manifold. 

To understand the analytic index map $\Ind_a$, we consider what we call a \emph{smooth symbol} $(V_1, V_2, \sigma)$ representing a class $[V_1, V_2, \sigma]$ in $K_\bullet\big(\cea\big)$.  Roughly a smooth symbol consists of a pair of smooth vector bundles $V_1$, $V_2$ over $\sfG_0$ and a smooth endomorphism $\sigma$ between their pullbacks to $A^*$ that is invertible at infinity; so far this is analogous to the symbol data appearing in the classical Atiyah-Singer theorem, but in our non-compact setting we need an additional assumption on boundedness of the derivatives of $\sigma$ (see Definition \ref{dfn:geometry-K} for details).   Using groupoid pseudodifferential operators as introduced in \cite{mont-pier}, \cite{nwx:pseudo}, we consider a `quantization' $\{Q^\lambda(\sigma)\}_{\lambda\in [0,1]}$ of the symbol, and construct a multiplier $Q(\sigma)$ of the tangent groupoid Roe $C^*$-algebra. We prove the following theorem under a natural bounded geometry assumption; roughly, this says that all the curvature tensors on the target fibers of $\sfG$ have uniformly bounded derivatives, and injectivity radius bounded below  (see Assumption \ref{ass:kappa} for details).

 \begin{thm} (Theorem \ref{thm:k-theory}) Assume the Euclidean structure on $A$ satisfies Assumption \ref{ass:kappa}. Given a smooth symbol $(V_1, V_2, \sigma)$ representing an element in $K_0(C^*_e(A^*))$, the operator $Q(\sigma)$ defines a $K$-theory element $[Q(\sigma)]$ in $K_0(C^*_\cale (\calt \sfG))$ such that $ev_{0*}([Q(\sigma)])=[V_1, V_2, \sigma]$. 
\end{thm}

We apply our theorem to study positive scalar curvature metrics.  In particular, we prove the following result; this requires the same boundedness assumptions on the derivatives of the curvature tensors and injectivity radius as before (see Assumption \ref{ass:kappa} for details).
\begin{thm}
Let $\sfG$ be a Lie groupoid, and assume that the associated Lie algebroid $A$ has even rank, and is equipped with a spin structure such that the associated Riemannian metric $g_A$ satisfies Assumption \ref{ass:kappa}.  Then the triple $(\cals^+, \cals^-, \sigma(\slashed{D}_+))$ associated to the Dirac operator is a smooth symbol in the sense of Definition \ref{dfn:geometry-K}.  Moreover, if 
$g_A$ has uniformly positive scalar curvature, the analytic index of the K-theory class $[\cals^+, \cals^-, \sigma(\slashed{D}_+)]$ vanishes in $K_0(C^*_\cale(\sfG))$, i.e. 
\[
\Ind_a\Big([\cals^+, \cals^-, \sigma(\slashed{D}_+)]\Big)=0. \qedhere
\]
\end{thm}
Applying this theorem to the regular foliations $(\calm, \widetilde{\calf})$ and $(\widehat{M}, \widehat{\calf}) $, in Corollary \ref{cor:vanishing} we obtain obstructions to the existence of leafwise positive scalar curvature metric on $(M, \calf)$ as a generalization of the Lichnerowicz vanishing theorem, \cite{co86}, \cite{zhang}. \\

The article is organized as follows. In Section \ref{coarse subsec}, we discuss coarse structures on a locally compact Hausdorff groupoid, and show how they arise from suitable Riemannian structures on Lie groupoids (in particular, on holonomy groupoids).  In Section \ref{roe sec}, we introduce the groupoid Roe algebra associated to a locally compact groupoid with an appropriate coarse structure and Haar system; if the groupoid is in addition a Lie groupoid, we then use Connes' tangent groupoid to define an analytic index for longitudinal elliptic operators.  In Section \ref{quant sec}, we introduce the notion of a smooth symbol and use the groupoid pseudodifferential operator theory to study the analytic index of such a symbol.  Finally, in Section \ref{sc sec}, we apply our groupoid index theory to study (regular) foliations with leafwise positive scalar metrics. \\

\noindent{\bf Acknowledgements}: Tang's research and Willett's research are partially supported by the US NSF. Yao's research is partially supported by the CNSF. We would like to thank Jerome Kaminker, Guoliang Yu, and Weiping Zhang for helpful suggestions and encouragement. The first two authors would like to thank the Shanghai Center for Mathematical Sciences for hosting their visits, where the main part of the work was done.

\section{Coarse structure on Lie groupoids} \label{coarse subsec}

We fix some basic notation for groupoids. Let $\sfG\rightrightarrows \sfG_0$ be a locally compact groupoid with unit space $\sfG_0$.  We will always assume that such a groupoid $\sfG$ is paracompact and Hausdorff.  Let $s,t:\sfG\to \sfG_0$ be the source and target maps of $\sfG$, which we always assume to be open.  Throughout this paper, we think of an arrow as from the left to the right: thus a pair $(g,h)$ of arrows is composable if $t(g)=s(h)$, and the composition is $gh$.   The $s$-fibers and $t$-fibers of $x\in \sfG_0$ are denoted respectively by $\sfG_x:=s^{-1}(x)$ and $\sfG^x:=t^{-1}(x)$.  

If $\sfG$ is in addition a Lie groupoid then $s$ and $t$ are assumed to be submersions.  Thus $s$ and $t$  are open maps, and the $s$- and $t$- fibers are closed submanifolds of $\sfG$.  We will usually assume in this case that all $s$-fibers (and therefore all $t$-fibers) are connected.  The Lie algebroid of $\sfG\rightrightarrows \sfG_0$ is $A:=\ker{t_*}\subset T\sfG|_{\sfG_0}$.  Noting that the right action of $\sfG$ on itself permutes the $t$-fibers, sections of $A$ can be identified with right $\sfG$-invariant vector fields on $\sfG$ that are tangent to the $t$-fibers.  

We are mainly interested in Lie groupoids (and in fact, in holonomy groupoids) in this paper; nonetheless, some definitions and results are stated in more generality than this where it seems potentially useful for future applications, and where the extra generality causes no extra difficulties.

\subsection{Coarse structures and metrics}\label{metric subsec}

In \cite[Section 2]{hig-ped-roe:controlled}, Higson, Pedersen and Roe introduced a notion of {\em coarse structure} on a locally compact, Hausdorff groupoid $\sfG$ in the following way. 

\begin{defn}\label{def:entourage}
A \emph{coarse structure} on a locally compact, Hausdorff groupoid $\sfG$ is a collection $\mathcal{E}$ of open subsets $E$ of the set of arrows of $\sfG$, called \emph{entourages}, that have the following properties:
\begin{enumerate}
\item the inverse of any entourage is contained in an entourage;
\item the (groupoid) product of two entourages is contained in an entourage;
\item \label{uncon} the union of two entourages is contained in an entourage;
\item \label{propcon} entourages are proper: that is, for any entourage $E$ and any compact subset $C$ of the objects $\sfG_0$, $E\cap s^{-1}(C)$ and $E\cap t^{-1}(C)$ are relatively compact;
\item \label{covcon} the union of all the entourages is $\sfG$.
\end{enumerate}
If there is an entourage that contains the unit space $\sfG_0$, the coarse structure is called \emph{unital}.  Two coarse structures $\mathcal{E}$ and $\mathcal{F}$ are \emph{equivalent} if every element of $\mathcal{E}$ is contained in some element of $\mathcal{F}$ and vice versa.
\end{defn}

\begin{ex}\label{ex:pair1}
Say $X$ is a locally compact topological space, and $\sfG=X\times X$ is the associated pair groupoid.  Then an equivalence class of coarse structures on $\sfG$ in the sense above is essentially the same thing as a coarse structure on $X$ that is compatible with the topology as described in \cite[Definition 2.22]{Roe:locg}.  More specifically, the difference between the definition of \cite[Definition 2.22]{Roe:locg} and the specialization of Definition \ref{def:entourage} above to pair groupoids is analogous to the difference between a maximal atlas and an atlas for a manifold; it does not make any substantial difference to the resulting theory.
\end{ex}

\begin{rem}\label{ex:comp}
Note that the collection of all open relatively compact subsets of $\sfG$ always constitutes a coarse structure.  Moreover, conditions \eqref{uncon} and \eqref{covcon} from Definition \ref{def:entourage} imply that every relatively compact subset of $\sfG$ is contained in some entourage, and thus (up to equivalence) the collection of relatively compact subsets is contained in any coarse structure.  

On the other hand, if the unit space $\sfG_0$ is compact, then condition \eqref{propcon} implies that all entourages in any coarse structure are relatively compact.    Hence if $\sfG_0$ is compact, any coarse structure on $\sfG$ is  equivalent to the coarse structure consisting of all open relatively compact subsets of $\sfG$.  Thus (the equivalence class of) a coarse structure gives no more information than the topology of $\sfG$ when $\sfG_0$ is compact.
\end{rem}

In the remainder of this subsection we assume that $\sfG$ is a Lie groupoid with connected $t$-fibers, and will describe how to define a coarse structure on $\sfG$ starting with an appropriate Euclidean structure on the Lie algebroid $A$.  The following example is an important, and particularly straightforward, special case of the general construction.

\begin{ex}\label{rem:alg coarse}
Let $\sfG$ be a Lie groupoid, and $A$ the associated Lie algebroid, equipped with a Euclidean structure $\{\langle \ ,\ \rangle_x\}_{x\in \sfG_0}$ as above.  Then we may regard $A$ itself as a Lie groupoid with unit space $\sfG_0$, and all arrows having the same source and target.  Write points in $A$ as $(x,v)$ where $x\in X$ and $v$ is in the fiber $A_x$ over $x$, and define 
$$
E_r:=\{(x,v)\in A~|~\|v\|_x<r\},
$$
where $\|v\|_x:=\sqrt{\langle v,v\rangle_x}$.  Then the collection $\{E_r\}_{r>0}$ clearly defines a coarse structure on the Lie groupoid $A$.
\end{ex}

Fix a smooth Euclidean structure $\{\langle\ ,\ \rangle_x\}_{x\in \sfG_0}$ on the bundle $A$ with base space $\sfG_0$.  Let $s^*A$ be the pullback of $A$ to $\sfG$ via the source map.  For each fixed $x\in \sfG_0$ with associated $t$-fiber $\sfG^x$, the restriction $(s^*A)|_{\sfG^x}$ identifies canonically with the tangent bundle $T\sfG^x$; in this way, the Euclidean structure on $A$ gives rise to a Riemannian metric on each of the $t$-fibers $\sfG^x$.  Moreover, these metrics are invariant under the right action of $\sfG$ on itself: precisely, right multiplication by each $g\in G$ defines a diffeomorphism $\sfG^{s(g)}\to \sfG^{t(g)}$, and invariance means that these diffeomorphisms are all isometries.

For each $x\in \sfG_0$, let $d^x$ denote the length metric on the $t$-fiber $\sfG^x$ induced by the Riemannian metric defined above; note that as each $\sfG^x$ is assumed connected, each $d^x$ is everywhere finite.  Moreover, the family of metric spaces $(\sfG^x,d^x)$ is invariant for the right action of $\sfG$: precisely, for any $g\in \sf G$ and $g_1,g_2\in \sfG^{s(g)}$
\begin{equation}\label{right inv}
d^{s(g)}(g_1,g_2)=d^{t(g)}(g_1g,g_2g).
\end{equation}
Define now a function $\rho:\sfG\to [0,\infty)$ by 
\[
\rho(g):=d^{t(g)}(g, t(g)). 
\]
We will show that under one additional condition (discussed below), the sets $\{\rho^{-1}([0,r))~|~r>0\}$ define a coarse structure on $\sfG$.  The bulk of the proof of this is in the next two lemmas.

\begin{lem}\label{lem:rho-d}The function $\rho$ is subadditive and symmetric, meaning 
\[
\rho(gh)\leq \rho(g)+\rho(h)\qquad \text{ and } \qquad \rho(g^{-1})=\rho(g)
\]
for all $g,h\in \sfG$ for which these expressions make sense.
\end{lem}

\begin{proof}
The triangle inequality and the right invariance of the metric imply 
\begin{eqnarray*}
& \rho(gh) &= d^{t(gh)}(gh, t(gh))\leq d^{t(gh)}(gh, h)+d^{t(gh)}(h, t(h))\\
& &=d^{t(g)}(g, t(g))+d^{t(h)}(h, t(h))=\rho(g)+\rho(h)
\end{eqnarray*}
and the right-invariance again implies
\[
\rho(g^{-1})=d^{t(g^{-1})}(g^{-1},t(g^{-1}))=d^{t(g^{-1}g)}(g^{-1}g,t(g^{-1})g)=d^{t(g)}(t(g),g)=\rho(g). \qedhere
\]
\end{proof}

\begin{lem}\label{lem:rho-cont}
Let $r>0$.  For any $g\in \sfG$ with $\rho(g)<r$ there is an open neighborhood $U$ of $g$ in $\sfG$ such that $\rho(h)<r$ for all $h\in U$. 
\end{lem}

\begin{proof}
Let $\epsilon>0$ be such that $2\epsilon<r-\rho(g)$.  Fix a smooth path from $t(g)$ to $g$ in $\sfG^{t(g)}$ with length $l$ at most $r-\epsilon$ (this is possible by definition of the length metric $d^{t(g)}$).  As $t:\sfG\to \sfG_0$ is a submersion and this path is compact, we may cover it by finitely many open sets (for the topology on $\sfG$) that identify with product sets of the form $V\times W$, where $V$ is an open neighborhood of $t(g)$ in $\sfG_0$, and $W$ is an open subset of $\sfG^{t(g)}$.  Using smoothness of the Euclidean structure on $A$ and another compactness argument based on moving along the original path through these product neighborhoods, it is not difficult to see that there is a neighborhood $U$ of $g$ (for the topology on $\sfG$) of the same form such that for every $h\in U$ there is a path from $t(h)$ to $h$ in $\sfG^{t(h)}$ of length at most $l+\epsilon$.  This completes the proof.
\end{proof}
 
\begin{prop}\label{prop:groupoid-entourage}
Assume that the Riemannian manifolds $\sfG^x$ constructed above from a Euclidean structure on the Lie algebroid $A$ are all complete.  Then the collection
\begin{equation}
\label{dfn:e-r} E_{r}:=\{g\in \sfG~|~ \rho(g)<r\}
\end{equation}
defines a coarse structure on $\sfG$.  
\end{prop}

The assumption that each $\sfG^x$ is complete is not automatic (Example \ref{ex:pair2} below), but holds in the examples of most interest to us (Proposition \ref{prop:example-entourage} below). 

\begin{proof} 
Symmetry and subadditivity as in Lemma \ref{lem:rho-d} imply that $E_r^{-1}=E_r$ and the product of $E_r$ and $E_s$ is contained in $E_{r+s}$.  Each $E_r$ is open by Lemma \ref{lem:rho-cont}.  To see properness, note that completeness of each $\sfG^x$ implies that there is a globally defined, continuous exponential map $\text{exp}:A\to \sfG$.  For any compact subset $C$ of $\sfG_0$, we have that 
$$
t^{-1}(C)\cap E_r\subseteq \text{exp}(\{(x,v)\in A~|~x\in C,\|v\|_x\leq r\},
$$
and the set on the right is clearly compact.  The fact that $s^{-1}(C)\cap E_r$ is relatively compact follows by symmetry.
\end{proof}

\begin{ex}\label{rem:alg coarse2}
Going back to Example \ref{rem:alg coarse}, note that if we consider $A$ as a Lie groupoid, then the Lie algebroid is just $A$ again.  For each $x\in \sfG_0$, the metrics $d^x$ defined above just identify with the metric on the fiber $A_x$ defined by the norm $\|\cdot\|_x$, and thus the coarse structure defined on the groupoid $A$ by Proposition \ref{prop:groupoid-entourage} is the same as the one in Example \ref{rem:alg coarse}.
\end{ex}

\begin{ex}\label{ex:pair2}
Let $M$ be a connected Riemannian manifold, and $\sfG=M\times M$ the associated pair groupoid.  Then the Lie algebroid $A$ identifies naturally with the tangent bundle $TM$, and thus inherits a Euclidean structure.  The coarse structure on $\sfG$ defined by the process above then identifies with the coarse structure on $M$ induced by the original metric as in \cite[Example 2.5]{Roe:locg}.  Moreover, each $t$-fiber $\sfG^x$ is isometric to $M$.

In particular, completeness of the $t$-fibers $\sfG^x$ is not automatic in general (as it would be if $\sfG$ were a Lie \emph{group}).
\end{ex}

For regular foliations, the construction above specializes as follows.  This is the key example of this paper.

\begin{prop}\label{prop:example-entourage} 
Let $\calf$ be a regular foliation on a complete Riemannian manifold $M$.  Let $\sfG$ be the associated holonomy groupoid, which we assume is Hausdorff.

Identify the Lie algebroid $A$ with $\calf$, and equip it with the Euclidean structure defined by restricting the Riemannian metric from $M$.  Then each of the Riemannian manifolds $\sfG^x$ defined above is complete, so we get a coarse structure on $\sfG$ as in Proposition \ref{prop:groupoid-entourage}.
\end{prop}

\begin{proof}
We first claim that each leaf $L$ is complete when equipped with the length metric $d_L$ induced from the restriction of the Riemannian metric on $M$.  This is presumably well-known, but as we couldn't find a proof in the literature we provide one for the reader's convenience.  

Let then $(x_n)$ be a Cauchy sequence for $d_L$.  The metric $d_L$ dominates the restriction of the length metric $d_M$ on $M$ to $L$, whence $(x_n)$ is Cauchy for $d_M$ as well, and thus convergent to some $x\in M$ by completeness.  Let $U$ be an open foliation neighborhood (i.e. a neighborhood of $x$ on which the foliation is a product) of $x$ in $M$ such that for some $\epsilon>0$, the $\epsilon$-neighborhood (defined using $d_M$) of $\overline{U}$ is still a foliation neighborhood of $x$.  Then the metric space $(L\cap \overline{U},d_L)$ splits into (possibly infinitely many) connected components, each of which is closed in $\overline{U}$, and all of which are all at least $\epsilon$ apart from each other for the $d_L$ metric.  Hence all but finitely many of the $x_n$ are in the same connected component of $L\cap \overline{U}$, whence $x$ is also in this connected component, and so in particular in $L$.  Hence $(L,d_L)$ is complete as claimed.

To finish the argument, recall that for each $x\in M$ the $t$-fiber $\sfG^x$ is a connected covering space of the leaf $L$ through $x$, with the Riemannian metric pulled back from $L$.  As the covering map $\sfG^x\to L$ is a local isometry for the induced length metrics, $\sfG^x$ is also complete.
\end{proof}

\begin{ex}\label{ex:s1-action}
Apart from regular foliations, actions of Lie groups are another interesting source of examples.  We give a basic example here to illustrate some issues that arise, but no doubt much more could be said.  Let $SO(2)$ act on $\R^2$ by rotations in the usual way, and let $\sfG:=\R^2\rtimes SO(2)$ be the associated crossed product groupoid; we write points of the crossed product as pairs $(x,z)$, where $x\in \R^2$ and $z\in SO(2)$.  One natural definition of a coarse structure on $\sfG$ is just to take $\cale_c:=\{\sfG\}$, which reflects the fact that the acting group $SO(2)$ is compact (`$_c$' is for `compact'), and so itself has no interesting coarse geometry.  On the other hand, identifying $SO(2)$ with the unit circle in $\C$, we can define a different coarse structure $\cale_m$ to consist of the sets 
$$
E_r:=\{(x,z)\in \sfG\mid |z-1|_\C\cdot |x|_{\R^2}<r\}
$$ 
for all $r>0$; the coarse structure $\cale_m$ reflects the fact that the orbits of the $SO(2)$ action get larger as one moves away from the origin in $\R^2$ (`$_m$' is for `metric').  The coarse structures $\cale_m$ and $\cale_c$ are not  equivalent, and have quite different properties.  

Neither $\cale_c$ nor $\cale_m$ arises directly from a Euclidean structure on the associated Lie algebroid $A$ as in Proposition \ref{prop:groupoid-entourage}, but both are equivalent to coarse structures arising in that way for appropriate choices of Euclidean structure on $A$.  Indeed, identify $A$ with $\R^2\times \mathfrak{so(2)}$, where $\mathfrak{so}(2)$ is the Lie algebra of $SO(2)$, a copy of $\R$.  Fixing a metric on $\mathfrak{so}(2)$, a Euclidean structure on $A$ is essentially the same thing as a choice of smooth function $s:\R^2\to (0,\infty)$, which governs how much the fixed metric on $\mathfrak{so}(2)$ is `scaled'.  The coarse structure $\cale_c$ is equivalent to that arising from the constant scaling factor $s(x)=1$, while the coarse structure $\cale_m$ is equivalent to that arising from the scaling factor $s(x)=1+|x|^2$.
\end{ex}

\begin{rem}\label{rmk:metric} In order to define a coarse structure on $\sfG$, it is sufficient to specify a smooth symmetric family of proper metrics on $\{\sfG^x\}_{x\in \sfG_0}$. This observation allows one to define coarse structures on general Lie groupoids that are not even source connected; however, to keep in contact with our main motivation --- regular foliations --- we will not use this.
\end{rem}

\section{Groupoid Roe algebras, tangent groupoids, and the analytic index} \label{roe sec}

\subsection{Groupoid Roe algebra} 

We assume throughout this subsection that $\sfG$ is a locally compact, second countable, Hausdorff groupoid equipped with a coarse structure.  The reader should bear in mind the case that $\sfG$ is a pair groupoid as in Example \ref{ex:pair2}, or that $\sfG$ is a holonomy groupoid with coarse structure defined using a metric on the underlying manifold $M$ as in Proposition \ref{prop:example-entourage}.

\begin{defn}\label{def:haar}
A \emph{Haar system} on $\sfG$ is a family $\{\mu^x\}_{x\in \sfG_0}$ satisfying the following conditions:
\begin{enumerate}
\item each $\mu^x$ is a regular Borel measure on the corresponding $t$-fiber $\sfG^x$ with full support;
\item the family is continuous, meaning that for each $f\in C_c(\sfG)$, the function 
$$
\sfG_0\to \C,\quad x\mapsto \int_{\sfG^x}f(g)d\mu^x(g) 
$$
is continuous\footnote{If $\sfG$ is a Lie groupoid, the family should be assumed smooth: just replace continuous functions by smooth functions everywhere.};
\item the family is right invariant, meaning that for any $g\in \sfG$ and any $f\in C_c(G)$,
$$
\int_{\sfG^{t(g)}} f(h)d\mu^{t(g)}(h)=\int_{\sfG^{s(g)}} f(hg)d\mu^{s(g)}(h).
$$
\end{enumerate}
\end{defn}

\begin{ex}\label{ex:rie haar}
If $\sfG$ is a Lie groupoid, Haar systems always exist.  Indeed, take any (smooth) Euclidean structure on the Lie algebroid $A$, lift it to a metric on each $t$-fiber as in the previous section, and define $\mu^x$ to be the measure on $\sfG^x$ canonically defined by the Riemannian metric.  We will use this choice for $\mu^x$ whenever it is convenient.
\end{ex}

\begin{defn}\label{def:roe}
Let $\sfG$ be a locally compact, Hausdorff groupoid equipped with a coarse structure and Haar system.  The \emph{Roe $*$-algebra} of $\sfG$, denoted $C_\mathcal{E}(\sfG)$, has as underlying vector space the collection of  continuous, bounded, complex-valued functions on $\sfG$ that are supported in an entourage.  The adjoint and multiplication on $C_\mathcal{E}(\sfG)$ are defined by
$$
f^*(g):=\overline{f(g^{-1})},\quad (f_1*f_2)(g):=\int_{\sfG^{t(g)}} f_1(gh^{-1})f_2(h)d\mu^{t(g)}(h).
$$
\end{defn}

It follows readily from the properties in Definitions \ref{def:entourage} and \ref{def:haar} that $C_\mathcal{E}(\sfG)$ is a well-defined $*$-algebra.  Indeed, note first that the adjoint is clearly a well-defined map, while condition \eqref{propcon} from Definition \ref{def:entourage} implies that the integrand appearing in the multiplication formula is integrable, so this formula also make sense.  On the other hand, the formulas defining the multiplication and adjoint above are the standard ones used to define groupoid convolution $*$-algebras (cf. \cite[Section II.1]{Renault:gpd approach}).  One can therefore see that the $*$-algebra axioms are satisfied by the same proofs that work for $C_c(\sfG)$: see for example \cite[Proposition II.1.1]{Renault:gpd approach}.

In order to define a $C^*$-algebraic completion of $C_{\mathcal{E}}(\sfG)$, we will need to impose extra conditions on the Haar system as follows.

\begin{defn}\label{def:bg}
Say $\sfG$ is a locally compact, Hausdorff groupoid with a coarse structure $\mathcal{E}$.  A Haar system $\{\mu^x\}$ for $\sfG$ has \emph{bounded geometry} if for all entourages $E$
$$
\sup_{x\in \sfG_0}\mu^x(E\cap \sfG^x)
$$
is finite.
\end{defn}

\begin{ex}\label{ex:rie bg}
Let $\sfG$ be a Lie groupoid with a fixed Euclidean structure on the Lie algebroid $A$.  Use this to equip each $\sfG^x$ with a Riemannian metric as in the previous section, which we assume complete, and let $\sfG$ have the associated coarse structure as in Proposition \ref{prop:groupoid-entourage}.  Let $\mu^x$ be the associated measure on $\sfG^x$ as in Example \ref{ex:rie haar}.   Assume that there is a global lower bound (possibly negative) on the Ricci curvatures of the Riemannian manifolds $\sfG^x$, independently of $x$.  Then the Bishop-Gromov theorem (see for example \cite[Theorem 107]{be03}) implies that the Haar system $\{\mu^x\}$ has bounded geometry.

In particular, say $M$ is a complete Riemannian manifold with Ricci curvature bounded below, and say $\sfG$ is the holonomy groupoid associated to some regular foliation on $M$.  Equip each $\sfG^x$ with the associated Riemannian metric, and $\sfG$ with the corresponding coarse structure, as in Proposition \ref{prop:example-entourage}.  Then the family of measures $\{\mu^x\}$ associated to the Riemannian metrics has bounded geometry.
\end{ex}

\begin{ex}\label{ex:disc bg}
Say $X$ is a locally finite discrete metric space, and $\sfG:=X\times X$ the associated pair groupoid equipped with the coarse structure defined by 
$$
E_r:=\{(x,y)\in X\times X~|~d(x,y)<r\}.
$$
Taking $\mu^x$ to be the counting measure on $\sfG^x$ for all $x\in \sfG_0$ defines a Haar system.  This Haar system has bounded geometry if and only if for any $r>0$, there is a uniform bound on the cardinalities of all $r$-balls in $X$.  This is the usual definition of bounded geometry in the setting of discrete metric spaces, and one motivation for the terminology we have adopted.
\end{ex}

The following definition is the analogue in our setting of that on \cite[page 50]{Renault:gpd approach} for the algebra $C_c(\sfG)$.

\begin{defn}\label{I norm}
Let $\sfG$ be a locally compact, Hausdorff groupoid equipped with a coarse structure and bounded geometry Haar system.
The \emph{$I$-norm} on $C_\mathcal{E}(\sfG)$ is defined by
$$
\|f\|_I:=\max\Big\{\sup_{x\in \sfG_0}\int_{\sfG^x}|f(g)|d\mu^x(g),~\sup_{x\in \sfG_0}\int_{\sfG^x}|f(g^{-1})|d\mu^x(g)\Big\}.
$$
\end{defn}
\noindent As elements of $C_{\mathcal{E}}(\sfG)$ are bounded and continuous, and as inverses of entourages are contained in entourages, the bounded geometry assumption implies that $\|f\|_I$ is finite.  Moreover, as each $\mu^x$ has full support, the $I$-norm is an honest norm rather than a semi-norm.

We are now ready to define the family of representations that we will use to make $C_\mathcal{E}(\sfG)$ into a $C^*$-algebra: these are based on the family of regular representations of a groupoid algebra $C_c(\sfG)$ as in \cite[Section 2.3.4]{Renault:notes}.

\begin{lem}\label{def:reps}
Let $\sfG$ be a locally compact Hausdorff groupoid equipped with a coarse structure and bounded geometry Haar system.  Fix $x\in \sfG_0$ and let $L^2(\sfG^x,\mu^x)$ be the corresponding $L^2$ space with norm $\|\cdot\|_2$.  For $\xi\in C_c(\sfG^x)$, $f\in C_{\mathcal{E}}(\sfG)$, and $g\in \sfG^x$, define 
$$
(\pi^x(f)\xi)(g):=\int_{\sfG^x}f(gh^{-1})\xi(h)d\mu^x(h).
$$
Then $\pi^x(f)$ satisfies the norm estimate
$$
\|\pi^x(f)\xi\|_2\leq \|f\|_I\|\xi\|_2
$$
and in particular extends to a bounded operator on $L^2(\sfG^x,\mu^x)$.  Moreover, $\pi^x$ defines a $*$-representation of $C_\cale(\sfG)$ on $L^2(\sfG^x,\mu^x)$.  
\end{lem}

\begin{proof}
Let $f$ be an element of $C_{\mathcal{E}}(\sfG)$ and $\xi$ be an element of $C_c(\sfG^x)$.  Then 
\begin{equation}\label{start}
\|\pi^x(f)\xi\|_2^2 = \int_{\sfG^x}\Big|\int_{\sfG^x} f(gh^{-1})\xi(h)d\mu^x(h)\Big|^2d\mu^x(g).
\end{equation}
Moving the absolute value inside the inner integral and applying Cauchy-Schwarz to the product 
$$
|f(gh^{-1})\xi(h)|=|f(gh^{-1})|^{1/2}|f(gh^{-1})|^{1/2}|\xi(h)|
$$
gives
\begin{align*}
\Big|\int_{\sfG^x} f(gh^{-1})\xi(h)d\mu^x(h)\Big|\leq \Big( \int_{\sfG^x}|f(gh^{-1})|d\mu^x(h)\Big)^{1/2}\Big(\int_{\sfG^x}|f(gh^{-1})|\cdot |\xi(h)|^2d\mu^x(h)\Big)^{1/2}.
\end{align*}
Using right invariance of the Haar system, the first factor is bounded above by $\|f\|_I^{1/2}$, and substituting into line \eqref{start} gives
$$
\|\pi^x(f)\xi\|_2^2\leq \|f\|_I\int_{\sfG^x}\int_{\sfG^x}|f(gh^{-1})|\cdot |\xi(h)|^2d\mu^x(h)d\mu^x(g).
$$
Switching the order of integration and using right invariance again bounds this above by
\begin{align*}
\|f\|_I\int_{\sfG^x}|\xi(h)|^2\int_{\sfG^x}|f(gh^{-1})|d\mu^x(g)d\mu^x(h)\leq \|f\|_I\int_{\sfG^x}\|f\|_I|\xi(h)|^2d\mu^x(h)\leq \|f\|_I^2\|\xi\|^2_2,
\end{align*}
giving the desired bound.  The fact that $\pi^x$ defines a $*$-homomorphism follows from the same sort of computations that show that the operations on $C_{\mathcal{E}}(\sfG)$ or $C_c(\sfG)$ define a $*$-algebra structure: see for example \cite[Proposition II.1.1]{Renault:gpd approach}.
\end{proof}

The following definition now makes sense.

\begin{defn}\label{def:red roe}
Let $\sfG$ be a locally compact, second countable, Hausdorff groupoid equipped with a coarse structure and a bounded geometry Haar system.  The \emph{Roe $C^*$-algebra} of $\sfG$, denoted $C^*_{\mathcal{E}}(\sfG)$, is the completion of $C_\mathcal{E}(\sfG)$ for the norm\footnote{It is a norm rather than a semi-norm as the measures $\mu^x$ have full support.}
$$
\|f\|:=\sup\{\|\pi^x(f)\|~|~x\in \sfG_0\}.
$$
\end{defn}  

We conclude this subsection with some remarks and examples relating Definition \ref{def:red roe} to other constructions in the literature.

\begin{rem}\label{rem:com case}
Say $\sfG$ is as in the above definition, and that $\sfG_0$ is compact.  Then Remark \ref{ex:comp} implies (whatever the coarse structure is) that $C_\mathcal{E}(\sfG)=C_c(\sfG)$, and that the Roe $C^*$-algebra $C^*_\mathcal{E}(\sfG)$ is equal to the usual reduced groupoid $C^*$-algebra $C^*_r(\sfG)$.  Thus our Roe $C^*$-algebras only give something new when the unit space $\sfG_0$ is not compact.  

Note in general that $C_\mathcal{E}(\sfG)$ always contains $C_c(\sfG)$, and that the $C^*_\cale(\sfG)$-norm on $C_\cale(\sfG)$ restricts to the $C^*_r(\sfG)$-norm on $C_c(\sfG)$; thus $C^*_\cale(\sfG)$ always contains $C^*_r(\sfG)$ as a $C^*$-subalgebra.
\end{rem}

\begin{ex}\label{rem:usual Roe}
Say $\sfG=X\times X$ is the pair groupoid associated to a discrete bounded geometry metric space $X$, equipped with the coarse structure and Haar system from Example \ref{ex:disc bg}.  Then $C_{\mathcal{E}}(\sfG)$ identifies with the $*$-algebra of finite propagation, bounded kernels on $X$, often denoted $\C_u[X]$.  Moreover, all the representations $\pi^x$ are unitarily equivalent to the standard representation of this algebra on $l^2(X)$, and thus $\cEg$ can be canonically identified with the uniform Roe algebra $C^*_u(X)$ of \cite[Section 4.4]{Roe:locg}.

More generally, say $\sfG=M\times M$ is the pair groupoid associated to a complete Riemannian manifold $M$.  The associated Lie algebroid identifies with $TM$, whence it inherits a Euclidean structure from the Riemannian structure on $M$, and this defines a coarse structure on $\sfG$ as in Example \ref{ex:pair2}.  If we assume that $M$ has Ricci curvature bounded below, then the associated Haar system $\{\mu^x\}$ has bounded geometry.  It is not too difficult to see that $\cEg$ identifies canonically with the usual Roe algebra $C^*(M)$ of \cite[Definitions 3.3 and 3.4]{Roe:index}.
\end{ex}

\begin{rem}\label{rem:roe} 
As part of a study of exactness for groupoids, Anantharaman-Delaroche recently generalized the uniform Roe algebra of a discrete group to define the \emph{uniform $C^*$-algebra} $C^*_u(\sfG)$ of a general locally compact Hausdorff groupoid $\sfG$ with Haar system: see \cite[Definition 5.1]{ana:exact}.  Although both Definition \ref{def:red roe} above and \cite[Definition 5.1]{ana:exact} are generalizations of classical (uniform) Roe algebras, the ingredients involved, the end results, and the intended applications are all quite different: in particular, \cite[Definition 5.1]{ana:exact} is a general construction that does not assume the presence of a coarse structure of $\sfG$.
\end{rem}

\begin{rem}\label{rem:max}
The above completion of $C_{\mathcal{E}}(\sfG)$ is an analogue of the reduced completion of the classical convolution algebra $C_c(\sfG)$.  It might be interesting to consider other completions, for example the natural maximal completion as was done in \cite{gwy:max} for classical Roe algebras.  We do not currently have any applications of this, so do not pursue it here.
\end{rem}

\begin{rem}\label{rem:sty}
Skandalis, Tu, and Yu \cite{sty:cbc gpd} (see also Tu's generalization \cite{tu:cbc gpd II}) constructed a \emph{coarse groupoid} $G(X)$ associated to a discrete bounded geometry metric space as in Example \ref{rem:usual Roe} above.  It is not too difficult to check that $C_{\mathcal{E}}(\sfG)$ identifies canonically with $C_c(G(X))$ in this case.  It seems reasonable to expect that there should be an analogous construction of a `coarse groupoid' associated to $\sfG$ with a coarse structure as in our current setting, but we did not pursue this. 
\end{rem}

\subsection{Tangent groupoid and associated Roe algebra}\label{tan gpd sec}

In this subsection, we recall the construction of the tangent groupoid \cite[Section II.5]{connes:book}, and build a coarse version of the associated $C^*$-algebra.  Throughout this section, $\sfG$ is a Lie groupoid and $A$ the associated Lie algebroid.  $A$ is assumed equipped with a Euclidean structure, and $\sfG$ and $A$ with the associated coarse structure and bounded geometry Haar system as in Proposition \ref{prop:groupoid-entourage}, Example \ref{rem:alg coarse2} and Example \ref{ex:rie haar}.  

Let $S$ be a closed submanifold of a manifold $X$ and let $\nu(S):=TX|_S/TS$ be its normal bundle. One can define a manifold with boundary $B_{S\hookrightarrow X}$ to be the disjoint union
\[
B_{S\hookrightarrow X}:=\{0\}\times \nu(S)\cup (0,1]\times X
\]
as a set, and use a choice of exponential map $\operatorname{exp}:\nu(S)\to X$ to produce a natural smooth structure on $B$ that restricts to those on $\nu(S)$ and $(0,1]\times X$, c.f. \cite[Section 3.1]{Hil-Sk:morphism} (the choices make no difference to the resulting smooth structure). 

We apply this general construction to a Lie groupoid $\sfG$ and the closed submanifold $\sfG_0$. Observe that the normal bundle $\nu(\sfG_0)$ is isomorphic to the Lie algebroid $A$ as a vector bundle over $\sfG_0$. We obtain the \emph{tangent groupoid} of $\sfG$, which as a set is given by 
\[
\mathcal{T}\sfG:=\{0\}\times A\cup (0,1]\times \sfG. 
\]
The manifold $\mathcal{T}\sfG$ can be viewed as a smooth family of groupoids over the interval $[0,1]$. Over the point $0\in [0,1]$, the fiber is $A$ viewed as a bundle of vector groups, and over $\lambda\in (0,1]$, the fiber is the groupoid $\sfG$. These groupoid structures are compatible with the smooth structure on $\mathcal{T}\sfG$; thus $\mathcal{T}\sfG$ is a Lie groupoid over the unit space $\sfG_0\times[0,1] $. 

Let $A_{\calt \sfG}$ denote the Lie algebroid of $\calt\sfG$, a bundle over $\sfG_0\times [0,1]$; it is isomorphic to the Lie algebroid $A\times [0,1]$, but in a slightly unnatural way, so we will not use this.  For each $\lambda\in [0,1]$, write $A_\lambda$ for the restriction of $A_{\calt \sfG}$ to $\sfG_0\times \{\lambda\}$, so each $A_\lambda$ identifies naturally with a copy of $A$.  For $\lambda\in (0,1]$, equip $A_\lambda$ with the original Euclidean structure rescaled by a factor of $1/\lambda^2$, and equip $A_0$ with the original Euclidean structure; thus the metric is `blown up' as $\lambda$ tends to zero.  These structures fit together to define a smooth Euclidean structure on $A_{\calt\sfG}$.   We use this metric structure to equip $\mathcal{T}\sfG$ with the coarse structure and Haar system from Proposition \ref{prop:groupoid-entourage} and Example \ref{ex:rie haar}.

The following lemma is clear from the definitions: the basic points are that $t$-fibers of $\mathcal{T}\sfG$ (equipped with metric and measure structure) identify canonically with $t$-fibers of either $\sfG$ or $A$, and that the `blow-up' of the metric is exactly compensated for by the corresponding blow-ups of volume.

\begin{lem}\label{lem:coarse comp}
The Lie algebroid $A$ always has bounded geometry with the coarse structure from Example \ref{rem:alg coarse} and Haar structure from Example \ref{ex:rie haar}.  Moreover, if $\sfG$ has bounded geometry in the sense of Definition \ref{def:haar}, then so too does $\mathcal{T}\sfG$.   \qed
\end{lem}

As in Definition \ref{def:red roe} above, we thus get an associated Roe $C^*$-algebra $C^*_{\mathcal{E}}(\mathcal{T}\sfG)$.  Note that there are natural $*$-homomorphisms
$$
ev_0:C_\mathcal{E}(\mathcal{T}\sfG)\to C_\mathcal{E}(A),\quad ev_1:C_\mathcal{E}(\mathcal{T}\sfG)\to C_\mathcal{E}(\sfG),
$$
(where we have used the same notation for the coarse structures on $\sfG$, $\mathcal{T}\sfG$ and $A$) defined by restricting to the fibers over $0$ and $1$ respectively in the tangent groupoid and using the compatibility of the coarse structures involved.  It is clear from the definition of the norms on the associated Roe $C^*$-algebras that these homomorphisms extend to the completions, giving surjective $*$-homomorphisms. 
\begin{equation}\label{ev homs}
ev_0:C^*_{\mathcal{E}}(\mathcal{T}\sfG)\to C^*_{\mathcal{E}}(A),\quad ev_1:C^*_{\mathcal{E}}(\mathcal{T}\sfG)\to C^*_{\mathcal{E}}(\sfG).
\end{equation}

Our next task is to compute the kernel of $ev_0$.  To this end, we first need a technical lemma about the norm on $C^*_\mathcal{E}(A)$.

\begin{lem}\label{lem:A norm}
Let $B$ be any $C^*$-algebra completion of $C_\mathcal{E}(A)$ for a norm that is bounded above by the $I$-norm as in Definition \ref{I norm}.  Then the identity map on $C_{\mathcal{E}}(A)$ extends to a surjective $*$-homomorphism
$$
C^*_\mathcal{E}(A)\to B.
$$
\end{lem}

\begin{proof}
Let $I_{\mathcal{E}}(A)$ denote the completion of $C_\mathcal{E}(A)$ for the $I$-norm.  We will prove the lemma by showing that $C^*_\mathcal{E}(A)$ is the enveloping $C^*$-algebra of the Banach $*$-algebra $I_{\mathcal{E}}(A)$.  

Let $U$ be a relatively compact open subset of $\sfG_0$ over which $A$ is trivialized, and $A_U$ the restriction of $A$ to $U$, so $A_U$ is diffeomorphic to $U\times \R^d$ where $d$ is the fiber dimension of $A$.  Then it is straightforward to check that the closure of 
$$
\{f\in C_{\mathcal{E}}(A)~|~f|_{A\setminus A_U}=0\}
$$
inside $I_{\mathcal{E}}(A)$ identifies isometrically via the diffeomorphism mentioned above with $C_0(U,L^1(\R^d))$.  We will start by showing that the enveloping $C^*$-algebra of $C_0(U,L^1(\R^d))$ is $C_0(U,C^*(\R^d))$.  For this, it suffices to show that any (non-zero) irreducible $*$-representation $\pi$ of $C_0(U,L^1(\R^d))$ extends to $C_0(U,C^*(\R^d))$.

We first claim that there exists $x\in U$ such that for every open set $V$ containing $x$ there exists $f\in C_0(U,L^1(\R^d))$ supported in $V$ such that $\pi(f)\neq 0$.  Indeed, as $\pi$ is non-zero and continuous for the $I$-norm there must exist $g\in C_c(U,L^1(\R^d))$ such that $\pi(g)\neq 0$.  If the claim fails, then there is an open cover $\mathcal{U}$ of $U$ such that $\pi(f)=0$ whenever $f\in C_0(U,L^1(\R^d))$ is supported in an element of $\mathcal{U}$.  However, as the support of $g$ is compact, we may use a partition of unity to write $g$ as a finite sum of functions supported in elements of $\mathcal{U}$, which contradicts that $g$ is non-zero.

Let then $x\in U$ satisfy the condition in the claim above.  We claim next that the kernel of $\pi$ factors through the homomorphism $\epsilon_x:C_0(U,L^1(\R^d))\to L^1(\R^d)$ defined by evaluating at the fiber over $x$.  If not, then there is some $f\in C_0(U,L^1(\R^d))$ with $f(x)=0$, but $\pi(f)\neq 0$.  As $\pi$ is continuous for the $I$-norm, we may assume that there is a neighborhood $V\owns x$ such that $f$ is zero when restricted to $V$.  Let $g\in C_0(U,L^1(\R^d))$ be supported in $V$ and such that $\pi(g)\neq 0$.  As the supports of $f$ and $g$ are disjoint, the ideals they generate in $C_0(U,L^1(\R^d))$ are orthogonal.  It is therefore impossible that both $\pi(f)$ and $\pi(g)$ are non-zero by irreducibility of $\pi$; this contradiction establishes the clam.

To complete the proof that the enveloping algebra of $C_0(U,L^1(\R^d))$ is $C_0(U,C^*(\R^d))$, note that we now know that $\pi$ factors as a composition $\rho\circ \epsilon_x$, where $\rho$ is an irreducible continuous $*$-representation of $L^1(\R^d)$.  The result follows as the enveloping $C^*$-algebra of $L^1(\R^d)$ is $C^*(\R^d)$, by definition of the latter algebra.\\

Continuing, let $V$ be an open subset of $\sfG_0$ given as a disjoint union $V=\bigsqcup_{i\in I}U_i$ where each $U_i$ is an open relatively compact subset of $\sfG_0$ over which $A$ is trivialized.  Let $A_V$ denote the restriction of $A$ to $V$, so $A_V$ is diffeomorphic to $V\times \R^n$.  Then the $I$-norm completion of 
$$
\{f\in C_{\mathcal{E}}(A)~|~f|_{A\setminus A_V}=0\}
$$
identifies with a Banach $*$-subalgebra, say $B_V$, of the Banach $*$-algebra
$$
\prod_{i\in I}C_0(U_i,L^1(\R^d))
$$
of bounded sequences of elements of the different $C_0(U_i,L^1(\R^d))$ with the norm given by the supremum over $i\in I$ of the norms on the individual $C_0(U_i,L^1(\R^d))$.  Let $\pi$ be any $*$-representation of $B_V$, and let $\pi_i$ denote the restriction of $\pi$ to the $i^\text{th}$ factor.  Then orthogonality of the factors implies that for any $(f_i)\in B_V\subseteq \prod_{i\in I}C_0(U_i,L^1(\R^d))$, 
$$
\|\pi((f_i))\|= \sup_{i\in I}\|\pi_i(f_i)\|\leq \sup_{i\in I}\|f_i\|_{C_0(U_i,C^*(\R^d))}
$$
by what we have already proven.  Hence the enveloping $C^*$-algebra of $B_V$ has norm dominated by the norm from
$$
\prod_{i\in I}C_0(U_i,C^*(\R^d)),
$$
which is in turn clearly equal to the norm $B_V$ inherits as a $*$-subalgebra of $C^*_{\mathcal{E}}(A)$.

To complete the proof, let $f\in I_{\mathcal{E}}(A)$ be arbitrary, and $\pi$ be a $*$-representation of $I_{\mathcal{E}}(A)$.  Say $\sfG_0$ has dimension $m$.  Then using paracompactness, we may write $\sfG_0$ as a union $V_0\cup\cdots \cup V_m$, where each $V_i$ has the same properties as $V$ above.  Using a partition of unity, we may write $f=\sum_{i=0}^m f_i$, where each $f_i$ is supported in $V_i$, and $\|f_i\|_{C^*_\mathcal{E}(A)}\leq \|f\|_{C^*_\mathcal{E}(A)}$ for each $i$.  Hence 
$$
\|\pi(f)\|\leq \sum_{i=0}^m \|\pi(f_i)\|\leq \sum_{i=0}^m \|f_i\|_{C^*_\mathcal{E}(A)}\leq (m+1)\|f\|_{C^*_\mathcal{E}(A)},
$$
where the second inequality follows from what we showed above about the enveloping $C^*$-norm on each $B_{V_i}$.  This shows every $*$-representation of $I_{\mathcal{E}}(A)$ is continuous for the $C^*_{\mathcal{E}}(A)$ norm, completing the proof.
\end{proof}

\begin{prop}\label{lem:ker-ev0} 
There is a canonical short exact sequence
$$
\xymatrix{ 0 \ar[r] & C_0(0,1]\otimes C^*_{\mathcal{E}}(\sfG)\ar[r] & C^*_{\mathcal{E}}(\mathcal{T}\sfG) \ar[r]^{ev_0} & C^*_{\mathcal{E}}(A)\ar[r] & 0 }
$$
(where `$\ \otimes$' denotes the spatial tensor product of $C^*$-algebras).
\end{prop}

\begin{proof}
Write $I_{\mathcal{E}}(\mathcal{T}\sfG)$ for the completion of $C_{\mathcal{E}}(\mathcal{T}\sfG)$ in the $I$-norm, and similarly write $I_{\mathcal{E}}(\sfG)$ for the completion of $C_{\mathcal{E}}(\sfG)$ in the $I$-norm, and $I_{\mathcal{E}}(A)$ for the completion of $C_{\mathcal{E}}(A)$ in the $I$-norm.  Then it is straightforward to see that there is a short exact sequence
\begin{equation}\label{alg ses}
\xymatrix{ 0 \ar[r] & C_0((0,1],I_{\mathcal{E}}(\sfG))\ar[r] & I_{\mathcal{E}}(\mathcal{T}\sfG) \ar[r]^{ev_0} & I_{\mathcal{E}}(A)\ar[r] & 0 },
\end{equation}
of Banach algebras, where the ideal is defined to consist of continuous functions from $[0,1]$ to $I_{\mathcal{E}}(\sfG)$ that vanish at $0$.  It is clear that the $C^*$-algebra norm the ideal $C_0((0,1],I_{\mathcal{E}}(\sfG))$ inherits from $C^*_{\mathcal{E}}(\mathcal{T}\sfG)$ is given by
$$
\|f\|:=\sup_{\lambda\in(0,1]}\|f(\lambda)\|_{C^*_{\mathcal{E}}(\sfG)},
$$
and thus that the closure of the ideal $C_0((0,1],I_{\mathcal{E}}(\sfG))$ inside $C^*_{\mathcal{E}}(\mathcal{T}\sfG)$ identifies canonically with $C_0(0,1]\otimes C^*_{\mathcal{E}}(\sfG)$.

We thus have a commutative diagram
$$
\xymatrix{ 0 \ar[r] & C_0(0,1]\otimes C^*_{\mathcal{E}}(\sfG)\ar[r] \ar[d]& C^*_{\mathcal{E}}(\mathcal{T}\sfG) \ar[r] \ar@{=} [d]& \frac{C^*_{\mathcal{E}}(\mathcal{T}\sfG)}{C_0(0,1]\otimes C^*_{\mathcal{E}}(\sfG)}\ar[r] \ar[d] & 0 \\ 
0 \ar[r] & J\ar[r] & C^*_{\mathcal{E}}(\mathcal{T}\sfG) \ar[r]^{ev_0} & C^*_{\mathcal{E}}(A)\ar[r] & 0 },
$$
where $J$ is the kernel of the quotient map $ev_0:C^*_{\mathcal{E}}(\mathcal{T}\sfG) \to \cEa$, the left hand vertical map is an injection, and the right hand vertical map a surjection.  Note that the right hand vertical map extends the identity on the dense $*$-subalgebra $C_{\mathcal{E}}(A)$ of the algebras involved, by the existence of the diagram in line \eqref{alg ses}.  As both $\frac{C^*_{\mathcal{E}}(\mathcal{T}\sfG)}{C_0(0,1]\otimes C^*_{\mathcal{E}}(\sfG)}$ and $C^*_{\mathcal{E}}(A)$ are $C^*$-algebra completions for norms on $C_{\mathcal{E}}(A)$ that are dominated the $I$-norm, the surjection must be an isomorphism by Lemma \ref{lem:A norm}.  Hence the left-hand vertical map is also an isomorphism, and we are done.
\end{proof}

Using the $K$-theory exact sequence and Lemma \ref{lem:ker-ev0}, we conclude this subsection with the following corollary.

\begin{cor}\label{prop:k-ext} 
The map
\[
{ev_0}_*: K_\bullet(C^*_{\mathcal{E}}(\mathcal{T}\sfG))\to K_\bullet(C^*_{\mathcal{E}}(A))
\]
induced on $K$-theory by the evaluation-at-zero map is an isomorphism. \qed
\end{cor}

\subsection{The analytic assembly map}

We keep the notation from the previous section: $\sfG$ is a Lie groupoid with $\mathcal{T}\sfG$ the associated tangent groupoid and $A$ the associated Lie algebroid; moreover, $A$ is equipped with a Euclidean structure making each $t$-fiber $\sfG^x$ a complete Riemannian manifold, and all three of these objects are equipped with the associated metric, coarse, and Haar structures.  We assume moreover that the Haar structure is of bounded geometry.   Throughout this section, we also write $A^*$ for the dual vector bundle of $A$, and equip it with the Euclidean structure induced from that on $A$ (the Euclidean structure can be used to identify $A$ and $A^*$, of course, but it is perhaps more natural to keep the distinction). 

Our first task in this section is to identify the Roe $C^*$-algebra $C^*_{\mathcal{E}}(A)$ with an algebra of functions on $A^*$.  Let $C_e(A^*)$ be the $C^*$-algebra of continuous bounded functions from $A^*$ to $\C$ such that:
\begin{enumerate}
\item for each $x\in \sfG_0$, the restriction $f_x$ of $f$ to the fiber $A^*_x$ over $x$ vanishes at infinity;
\item the family $\{f_x\}_{x\in \sf G_0}$ of restrictions is `equicontinuous' in the sense that for any $\epsilon>0$ there exists $\delta>0$ such that if $\xi,\eta\in A^*_x$ satisfy $\|\xi-\eta\|_x<\delta$, then $|f_x(\xi)-f_x(\eta)|<\epsilon$.  
\end{enumerate}

\begin{lem}\label{lem:fourier}
The Roe $C^*$-algebra $C^*_{\mathcal{E}}(A)$ is isomorphic to $C_e(A^*)$.
\end{lem}

\begin{proof}
For each $x\in \sfG_0$ recall that $A_x$ denotes the fiber of $A$ over $x$, and let
$$
\calf:C^*(A_x)\to C_0(A^*_x),\quad f\mapsto \calf(f)
$$
denote the Fourier isomorphism.  
For each $r>0$, let $C_{\mathcal{E},r}(A)$ denote the subspace of $C_{\mathcal{E}}(A)$ consisting of functions supported in  
$$
A^{(r)}:=\{(x,v)\in A~|~x\in \sfG_0,~v\in A_x,~\|v\|_x< r\}.
$$ 
Let $C_{b}(A^*)$ denote the $C^*$-algebra of bounded continuous functions on $A^*$.  Then the formula
$$
\{f_x\}_{x\in \sfG_0}\mapsto \{\calf(f_x)\}_{x\in \sfG_0}
$$
defines a $*$-homomorphism
$$
C_{\mathcal{E}}(A)\to C_{b}(A^*).
$$
Moreover, the fact that an element on the left is supported in some $A^{(r)}$ implies by standard Fourier theory that the family $\{\calf(f_x)\}$ is equicontinuous and vanishes at infinity, whence the image is actually in $C_e(A^*)$.  Standard facts about the Fourier transform imply moreover that this map is isometric for the norm induced on the left hand side by $\cEa$, and thus it extends to an injective $*$-homomorphism $\phi:\cEa\to \cea$.  It remains to show that the image of $\phi$ is dense. 

For this, let $f$ be an element of $C_e(A^*)$, which we write as a family $\{f_x\}_{x\in \sfG_0}$ of restrictions to fibers.  Let $d$ be the fiber dimension of $A$, and let $g$ be a smooth, positive function on $\R^d$ with integral one, that only depends on the norms of elements, and that has compactly supported inverse Fourier transform.  For $\delta>0$, define $g_\delta(\xi)=\delta^{-d}g(\delta \xi)$.  Note that as each $g_\delta$ depends only on the norm on $\R^d$, we may define a function $\widetilde{g_\delta}:A^*\to \C$ by the formula $(x,\xi)\mapsto g_\delta(\|\xi\|_x)$ in the obvious sense, and that $\widetilde{g_\delta}$ thus defined is an element of $C_e(A^*)$.

Now, note that for each $\delta>0$ we may define an element $f*\widetilde{g_\delta}$ of $C_e(A^*)$ by taking the fiberwise convolution.  Equicontinuity of the family $\{f_x\}$ implies that 
$$
\lim_{\delta\to 0}\|f*\widetilde{g_\delta}-f\|_{C_e(A^*)}=0.
$$
Let now $h_\delta:A\to \C$ be given fiberwise as the inverse Fourier transform of $f*\widetilde{g_\delta}$, and note that each $h_\delta$ is an element of $C_{\mathcal{E}}(A)$ by compact support of the inverse Fourier transform of $g$ (and hence of each $g_\delta$).  We then have
$$
\lim_{\delta\to 0}\|\phi(h_\delta)-f\|=\lim_{\delta\to 0}\|f*\widetilde{g_\delta}-f\|_{C_e(A^*)}=0,
$$
which shows that the image of $\phi$ is dense as claimed.  
\end{proof}
 
Lemma \ref{lem:fourier} shows that $C^*_{\mathcal{E}}(A)$ is isomorphic to $C_e(A^*)$ and thus ${ev_0}_*$ can be lifted to an isomorphism
\[
{ev_0}_*: K_\bullet(C^*_{\mathcal{E}}(\mathcal{T}\sfG))\to K_\bullet(C_e(A^*)). 
\]
Combining the above map with the evaluation map ${ev_1}_*$, we are finally able to define our analytic index map.

\begin{defn}\label{def:aind}
Let $\sfG$ be a groupoid equipped with the usual additional metric, coarse, and measure structures.  The \emph{analytic index map} is the homomorphism defined by
\begin{equation}\label{eq:ana-index}
\Ind_a:={ev_1}_*\circ {ev_0}_*^{-1}:K_\bullet(C_e(A^*))\to K_\bullet(C^*_{\mathcal{E}}(\sfG))).
\end{equation}
\end{defn}

\section{Quantization and index map}\label{quant sec}

Throughout this section, $\sfG$ is a Lie groupoid and $A$ the associated Lie algebroid.  $A$ is assumed equipped with a Euclidean structure, and $\sfG$ and $A$ with an associated coarse structure and bounded geometry Haar system as in Proposition \ref{prop:groupoid-entourage}, Example \ref{rem:alg coarse2} and Example \ref{ex:rie haar}.   Moreover, $\calt\sfG$ is the tangent groupoid of $\sfG$, equipped with the corresponding coarse and measure structures as in Section \ref{tan gpd sec}.

We will study the behaviour of the analytic index map of Definition \ref{def:aind} on elements of $K_0(C_e(A^*))$ with particularly good representatives as in Definition \ref{dfn:geometry-K}.  For such elements, we are able to compute the analytic index by the idea of quantization, and thus show that it agrees with a more classical notion of index defined using pseudo-differential operators.    Throughout, we work for simplicity with the case of $K_0$ only, although it should certainly be possible to extend the results to $K_1$.   

\subsection{Smooth Symbols for $K_0(C_e(A^*))$}\label{subsubsec:symbol}  

In order to carry out our quantization program, we will need to work with a restricted class of cycles for $K_0(C_e(A^*))$; we now describe some necessary preliminaries.  Compare the discussion in \cite[Section 4.2]{Pf-Po-Ta:lie-groupoid-index} for a treatment of similar geometric issues.

Recall that the Lie algebroid $A$ is equipped with a smooth Euclidean metric $\{\langle\ ,\ \rangle_x\}_{x\in \sfG_0}$, and that for each $x\in \sfG_0$, the tangent space $T\sfG^x$ to the $t$-fiber $\sfG^x$  is equipped with the associated Riemannian structure as discussed in Section \ref{metric subsec}.  Let $t_*:T\sfG\to T\sfG_0$ be the differential of the target map and define $T_t\sfG:=\text{ker}(t_*)$, a subbundle of $T\sfG$.  Let $T_t^*\sfG$ be the dual bundle to $T_t\sfG$; in other words, $T_t^*\sfG$ is the smooth bundle over $\sfG$ whose restriction to each $t$-fiber $\sfG^x$ is the cotangent bundle $T^*\sfG^x$.  Note that $\sfG$ acts freely and properly on the right of $T_t^*\sfG$, and the quotient identifies canonically with $A^*$; on the other hand, $A^*$ also identifies canonically with the restriction $T_t^*\sfG|_{\sfG_0}$ of $T_t^*\sfG$ to the unit space.  

Let $\tilde{t}:T_t^*\sfG\to \sfG_0$ be the map that takes each cotangent bundle $T^*\sfG^x$ to $x$, and note that $\tilde{t}$ is a smooth submersion as $t:\sfG\to \sfG_0$ is.  Let $\calf$ be the integrable subbundle of $TT_t^*\sfG$ corresponding to the regular foliation of $T_t^*\sfG$ by the fibers of $\tilde{t}$; in other words, each restriction $\calf|_{T^*\sfG^x}$ is the tangent bundle to $T^*\sfG^x$.  

Let now $W_1$, $W_2$ be smooth vector bundles over $A^*$ equipped with Euclidean structures, and let $\sigma:W_1\to W_2$ be a bundle endomorphism.  Using that the quotient of $T^*_t\sfG$ for the $\sfG$-action is $A^*$, we can lift this data to $\sfG$-equivariant bundles and a $\sfG$-equivariant bundle map $\widehat{\sigma}:\widehat{W}_1\to \widehat{W}_2$ over $T^*_t\sfG$; note that the Euclidean structures on $W_1$ and $W_2$ also lift to Euclidean structures on $\widehat{W}_1$ and $\widehat{W}_2$.  Let $\Gamma(V)$ denote the space of smooth sections of a vector bundle $V$.  For each $x$, note that $T^*\sfG^x$ inherits a Riemannian structure from $\sfG^x$, and let $\nabla^x$ be the associated Levi-Civita connection, so in particular $\nabla^x$ defines a map
$$
\nabla^x : \Gamma(\calf|_{T^*\sfG^x})\otimes \Gamma(\text{Hom}(\widehat{W}_1,\widehat{W}_2)|_{T^*\sfG^x})\to \Gamma(\text{Hom}(\widehat{W}_1,\widehat{W}_2)|_{T^*\sfG^x}),
$$
recalling that $\calf|_{T^*\sfG^x}=TT^*\sfG^x$.  Let $\nabla$ be the map
$$
\nabla : \Gamma(\calf)\otimes \Gamma(\text{Hom}(\widehat{W}_1,\widehat{W}_2))\to \Gamma(\text{Hom}(\widehat{W}_1,\widehat{W}_2))
$$
induced by combining all the maps $\nabla^x$.  We use $\nabla^k\widehat{\sigma}$  to denote the $k$th-order directional derivatives of the symbol $\widehat{\sigma}$, which is a bundle map
$$
\nabla^k \widehat{\sigma}:=\underbrace{\nabla \circ \cdots \circ \nabla }_k \widehat{\sigma}:\otimes ^k \calf\to \Hom(\widehat{W}_1, \widehat{W}_2)
$$ 
Let $\mathbb{S}(\otimes^k  \calf)$ denote the unit sphere bundle of $\otimes^k \calf$, and for $m\in \N$ define 
\[
N_m(\widehat{\sigma}):=\operatorname{sup}_{0\leq k\leq m} \sup_{v\in\mathbb{S}(\otimes^k  \calf)}\|\nabla^k\widehat{\sigma}(v)\|_{\Hom(\widehat{W}_1, \widehat{W}_2)}
\]
where the norm on $\Hom(\widehat{W}_1, \widehat{W}_2)$ is induced by the $\sfG$-invariant Euclidean structures on $\widehat{W}_1$ and $\widehat{W}_2$.

We now define the symbol data we will use.  Let $A^*_{\mathcal{E}}$ denote the Gelfand spectrum of the commutative $C^*$-algebra $C_e(A^*)$, which is a locally compact Hausdorff space.  Note that $C_e(A^*)$ contains $C_0(A^*)$ as an essential ideal, and therefore $A^*_{\mathcal{E}}$ contains $A^*$ as a dense open subset.   Let $\pi_{A^*}:A^*\to \sfG_0$ denote the bundle projection for $A^*$.  Abusing notation, if $V$ is a bundle over $\sfG_0$, we will write $\widehat{V}$ for the $\sfG$-equivariant bundle $\widehat{\pi_{A^*}^*V}$ over $T^*_t\sfG$ defined above.

\begin{defn}\label{dfn:geometry-K} 
A \emph{smooth symbol} for $C_e(A^*)$ is a quadruple $(V_1,V_2,\sigma,\tau)$ that satisfies the following conditions.
\begin{enumerate}
\item $V_1$ and $V_2$ are smooth complex vector bundles over $\sfG_0$ such that the pullbacks $\pi_{A^*}^*V_1$ and $\pi_{A^*}^*V_2$ are restrictions of bundles over the one point compactification of $A^*_{\mathcal{E}}$.
\item $\sigma:\pi_{A^*}^*V_1\to \pi_{A^*}^*V_2$ and $\tau:\pi_{A^*}^*V_2\to \pi_{A^*}^*V_1$ are bundle endomorphisms that come as restrictions of bundle endomorphisms on the one-point compactification of $A^*_{\mathcal{E}}$. 
\item Thinking of $\sigma$ and $\tau$ as bundle endomorphisms on the one-point compactification of $A^*_{\mathcal{E}}$, the compositions $\sigma\circ \tau$ and $\tau\circ \sigma$ are equal to the identity in some neighbourhood of the point at infinity.
\item The norms $N_m(\widehat{\sigma})$ and $N_m(\widehat{\tau})$ are finite for all $m$.
\end{enumerate}
\end{defn}

We will need the following remarks about these symbol classes.

\begin{rem}\label{rem:symbol}
Let $(V_1,V_2,\sigma,\tau)$ denote a smooth symbol for $C_e(A^*)$.
\begin{enumerate}
\item Using standard techniques in topological $K$-theory (compare for example \cite[Proposition (A. I)]{segal:k-theory} or the discussion in \cite[pages 491-492]{Atiyah-Singer:I}), we see that any smooth symbol $(V_1, V_2, \sigma,\tau)$ defines a class $[V_1, V_2, \sigma,\tau]$ in $K_0(C_e(A^*))$.    We do not know if any class in $K_0(C_e(A^*))$ can be represented by a smooth symbol, but this seems unlikely. 
\item The choice of `partial inverse' $\tau$ does not affect the resulting $K$-theory class; we will sometimes abusively write a smooth symbol as a triple $(V_1,V_2,\sigma)$ when the choice of $\tau$ does not matter.  
\item \label{hom rem} As $\sigma$ and $\tau$ are mutually inverse on some neighbourhood of the point at infinity, they are mutually inverse outside some neighbourhood of the zero section in $A^*$ which has compact closure in the one-point compactification of $A^*_{\mathcal{E}}$.  Replacing $\sigma$ and $\tau$ with functions that are homogeneous of degree zero on the fibers of $A^*$ outside of this neighbourhood does not affect the resulting $K$-theory class (compare \cite[pages 491-492]{Atiyah-Singer:I} again); we will sometimes assume homogeneity of this sort.
\item \label{comp rem} Let $\overline{\sfG_0}$ denote the closure of $\sfG_0$ in the one-point compactification of $A^*_{\mathcal{E}}$, so $\overline{\sfG_0}$ is some compact Hausdorff space.  Writing $V$ for $V_1$ or $V_2$, we know that $\pi_{A^*}^*V$ is the restriction of some bundle on the one-point compactification of $A^*_{\mathcal{E}}$, and therefore $V=\pi_{A^*}^*V|_{\sfG_0}$ is the restriction of some bundle on the compact space $\overline{\sfG_0}$.  In particular, there is a (smooth) bundle $W$ on $\sfG_0$ with $V\oplus W$ trivializable.  
\end{enumerate}
\end{rem}

\subsection{Quantization}
Let $(V_1,V_2,\sigma)$ be a smooth symbol with associated class $[V_1,V_2,\sigma]\in K_0(C_e(A^*))$.  Our goal is to describe a geometric realization of $\Ind_a([V_1, V_2, \sigma])$ by the idea of quantization. 

We will need to make the following assumptions on the family $\{\sfG^x\}_{x\in \sfG_0}$ of Riemannian manifolds, which will be in force throughout this subsection.  As the Riemannian metrics on this family are determined by the Euclidean structure on $A$, they are (indirectly) just assumptions about the latter structure.  

\begin{assump}\label{ass:kappa}
\begin{enumerate}
\item The associated metric spaces $(\sfG^x, d^x)$ are proper for every $x\in \sfG^0$.
\item There is $\iota>0$ such that for all $x\in \sfG_0$, the injectivity radius of $\sfG^x$ is bounded from below by $\iota$.
\item Let $R^\nabla_x$ be the Riemannian curvature tensor associated to the Riemannian metric on $\sfG^x$.  For any $k\geq 0$, there is $\kappa_k>0$ such that $k$th-order directional derivatives of $R^\nabla_x$ are bounded by $\kappa_k$ for all $x\in \sfG_0$, i.e.
\[
\sup_{v\in\mathbb{S}(\otimes^k  T\sfG^x)} ||\nabla^k R^\nabla_x(v)||_{\otimes ^4 T^*\sfG^x}\leq \kappa_k, \forall x\in \sfG_0
\]
(note that this implies via the discussion of Example \ref{ex:rie bg} that the Haar system on $\sfG$ has bounded geometry in the sense of Definition \ref{def:bg}).
\end{enumerate}
\end{assump}

The above Assumption \ref{ass:kappa} on uniformly bounded derivatives of the Riemannian curvature tensor implies that there is $\kappa>0$ such that for all $x\in \sfG_0$, all sectional curvatures of $\sfG^x$ are in $[-\kappa,\kappa]$. 

Choose $R>0$ and $r\in (0, \iota)$ such that $r+R<\iota$; these constants will be fixed throughout.  Lemma \ref{lem:cover} and Assumption \ref{ass:kappa} (more precisely the boundedness of the sectional curvature) imply that there is an integer $m>0$ such that for each $x\in \sfG_0$ there is an open cover $\mathcal{U}^x$ of $\sfG^x$ by balls of radius $r$ such that for any $U_0\in \mathcal{U}^x$ there are at most $m$ balls $U\in \mathcal{U}^x$  with 
\[
U\cap N_R(U_0)\ne \varnothing,
\]
where $N_R(U_0):=\{y\in \sfG^x\mid d(y, U_0)<R\}$.   Note that as $r<\iota$, every $U\in \mathcal{U}^x$ is diffeomorphic to $\R^d$, where $d$ is the dimension of $\sfG^x$.

For each $x\in \sfG_0$, consider the exponential map $\exp_x: A_x\to \sfG^x$ with respect to the Riemannian metric $\langle\ ,\ \rangle_x$.
Choose $R'$ with $\iota>R>R'>0$. We can find a smooth function $\chi: A\to \mathbb{R}_{\geq 0}$ which is zero outside $\{(x,v)\in A\mid \|v\|_x<R\}$ and equal to one on $\{(x,v)\in A\mid \|v\|_x<R'\}$. By the exponential map $\exp$, $\chi$ lifts to a smooth function $\tilde{\chi}$ on $\sfG$ with the property that for any $x\in \sfG_0$, $\tilde{\chi}|_{\sfG^x}$ is supported inside the $R$-ball centered at $x$ in the $t$-fiber $\sfG^x$.  Let $\Gamma(T\sfG^x)$ denote the sections of $T\sfG^x$ and let 
$$
\nabla^x:\Gamma(T\sfG^x)\otimes C^\infty(\sfG^x)\to C^\infty(\sfG^x)
$$ 
be the usual directional derivative operator on the Riemannian manifold $\sfG^x$; recalling that $T_t\sfG:=\ker(t_*:T\sfG\to T\sfG_0)$, let  
$$
\nabla:\Gamma(T_t\sfG)\otimes C^\infty(\sfG)\to C^\infty(\sfG)
$$
be defined by combining the various $\nabla^x$.  Then if $d$ is the dimension of $\sfG^x$ (equivalently, the rank of $A^*$), we can use the fact that $R$ and $R'$ are fixed constants and Assumption \ref{ass:kappa} on uniformly bounded derivatives of the Riemannian curvature tensor, to assume that $\nabla^k \tilde{\chi}$ is uniformly bounded on the unit sphere $\mathbb{S}(\otimes^k T_t\sfG)$ of $T_t\sfG$ for $k=1,\cdots, 2d+1$.

Given a smooth section $\xi$ of $A^*$,  define $e_\xi:\sfG\to\C$ by 
\begin{equation}\label{eq:funct-e}
e_\xi(g):=\tilde{\chi}(g)\exp\Big(2\pi \sqrt{-1}\langle\xi(t(g)), \exp^{-1}_{t(g)}(g)\rangle\Big)
\end{equation}
By the definition of $\tilde{\chi}$, $e_\xi$ is a smooth function on $G$ such that for each $x\in \sfG_0$, the restriction $e_\xi|_{\sfG^x}$ is supported within the $R$-neighborhood of $x\in \sfG^x$.

Let now $(V_1, V_2, \sigma)$ be a smooth symbol as is defined in Definition \ref{dfn:geometry-K}. 
With the above preparation, we are ready to define the quantization $Q^\lambda(\sigma)$.  Recall first the definitions of the metric and measure structures on $\calt\sfG$ from Section \ref{tan gpd sec}, which depend on $\lambda$.  We will write $\langle,\rangle_{x,\lambda}$ for the Euclidean structure on the fiber $A_{x,\lambda}$ of the Lie algebroid of $\calt \sfG$ over $(x,\lambda)\in \sfG_0\times [0,1]$, and $\mu^{x,\lambda}$ for the measure on the $t$-fiber $\calt\sfG^{\lambda,x}$.  Note that this $t$-fiber $\calt\sfG^{x,\lambda}$ is naturally a diffeomorphic a copy of the corresponding $t$-fiber $\sfG^x$ of $\sfG$ and we will often use this identification; however, the identification is \emph{not} isometric (or volume preserving).

\begin{defn}\label{defn:quantization} The quantization of a smooth symbol $(V_1, V_2, \sigma)$ is defined as follows.  For $i=1,2$, let $H_i^{x,\lambda}$ be the space of $L^2$-sections of the bundle $s^*V_i|_{\sfG^x}$ on $\sfG^x$, defined with respect to the measure $\mu^{x,\lambda}$. 

For a function $f\in C_\cale(\sfG)$, define a bounded operator $\rho^{x,\lambda}_i(f)$  $(i=1,2)$ on $H^{x,\lambda}_i$ by the formula 
\[
(\rho^{x,\lambda}_i(f)\xi)(g):=\int_{\sfG^x}f(gh^{-1})\xi(h)d\mu^{x,\lambda}(h)
\]
where $\xi$ is in $H^{x,\lambda}_i$ and $g\in \sfG^x$ (the same proof as of Lemma \ref{def:reps} above shows that this is well-defined and bounded independently of $x$ and $\lambda$).

For $(x,\lambda)\in \sfG_0\times(0,1]$ define $Q^{x,\lambda}(\sigma):H_1^{x,\lambda}\to H_2^{x,\lambda}$ by the formula
\[
Q^{x,\lambda}(\sigma)(\varphi)(g):=\frac{1}{(2\pi)^d}\int_{A^*_{s(g),\lambda}}\int _{\sfG^x}e^\lambda_{-\xi}(hg^{-1}) \sigma(s(g), \xi)\varphi(h)  d\mu^{x,\lambda}(h)d\xi,
\]
where $g\in \sfG^x$, $\varphi\in H_1^{x,\lambda}$, $d\xi$ is the volume element on $A^*_{s(g),\lambda}$ associated to the metric $\langle\ ,\ \rangle_{x,\lambda}$ on $A_{s(g),\lambda}$, and $e^\lambda_\xi$ is defined in the same way as Equation (\ref{eq:funct-e}) with the $\lambda$-scaled metric. (we will see below that the formula does indeed define a bounded operator).   
\end{defn}

To study the norm properties of the operators $Q^{x,\lambda}(\sigma)$, we will need the following result from \cite[Lemma 3.9]{so:quasi-class}.

\begin{lem}
\label{lem:sobolev} Let $p(x, y, \xi)$ be a continuous function on $\mathbb{R}^d\times \mathbb{R}^d\times \mathbb{R}^d$ such that 
\[
|\nabla^x_{i_1}\cdots \nabla^x_{i_s} \nabla^y_{j_1}\cdots \nabla^y_{j_t} \nabla^\xi_{k_1}\cdots \nabla^\xi_{k_u}p(x,y,\xi)|   
\] 
is bounded on $\mathbb{R}^d\times \mathbb{R}^d\times \mathbb{R}^d$ for 
\[
\ 0\leq i_1, \cdots, i_s, j_1, \cdots, j_t, k_1, \cdots, k_u\leq d,\ 0\leq s, t\leq k,\ 0\leq u\leq d+1,
\]
where $k=[d/2]+1$. Let $N(k,k, d+1)(p)$ be the norm of the function $p$ defined as follows
\[
N^{k, k, d+1}(p)=\max\limits_{\tiny \begin{array}{c}0\leq s\leq k\\ 0\leq t \leq k\\ 0\leq u\leq d+1\end{array}}\operatorname{sup}|(\nabla^x)^s(\nabla^y)^t(\nabla^\xi)^up(x, y, \xi)|. 
\]
In the above definition of $N^{k,k,d+1}(p)$, the term $|(\nabla^x)^s(\nabla^y)^t(\nabla^\xi)^up(x, y, \xi)|$ is defined by
\[
|(\nabla^x)^s(\nabla^y)^t(\nabla^\xi)^up(x, y, \xi)|^2=\sum_{\tiny\begin{array}{c}0\leq i_1, \cdots, i_s\leq d\\ 0\leq j_1, \cdots, j_t\leq d\\ 0\leq k_1, \cdots, k_u\leq d\end{array}} |\nabla^x_{i_1}\cdots \nabla^x_{i_s}\nabla^y_{j_1}\cdots \nabla^y_{j_t}\nabla^\xi_{k_1}\cdots \nabla^\xi_{k_u} p(x, y, \xi)|^2. 
\]
Let $\operatorname{Op}^\lambda(p)$ be a linear operator on $L^2(\mathbb{R}^d)$ defined by
\[
\operatorname{Op}^\lambda(p)(u):=\left(\frac{1}{2\pi \lambda}\right)^{d} \int_{\mathbb{R}^d}d\xi \int_{\mathbb{R}^d} dy e^{\frac{i\langle x-y, \xi \rangle}{\lambda}}p(x, y, \xi)u(y)
\]
for any function $u$ on $L^2(\mathbb{R}^d)$ of Schwartz type. 

The operator $\operatorname{Op}^\lambda(p)$ is a bounded operator on $L^2(\mathbb{R}^d)$ such that there is a constant $C>0$ such 
\[
\|\operatorname{Op}^\lambda(p)\|\leq C N^{k, k, d+1}(p),
\]
for $0<\lambda \leq 1$. The constant $C$ only depends on the dimension $d$. 
\end{lem}

\begin{prop}\label{lem:bound} Let $(V_1, V_2, \sigma)$ be a smooth symbol for $K_0(C_e(A^*))$.  Then under Assumption \ref{ass:kappa}, there is a positive constant $C>0$ such that for any $(x,\lambda)\in \sfG_0\times (0,1]$ the operator $Q^{x,\lambda}(\sigma):H_1^{x,\lambda}\to H_2^{x,\lambda}$ satisfies
\[
\|Q^{x,\lambda}(\sigma)\|_{ \mathcal{B}(H_1^{x,\lambda},H_2^{x,\lambda})}\leq C,
\] 
i.e.\ $Q^{x,\lambda}(\sigma)$ is uniformly bounded for all $\lambda$ and $x$. 
\end{prop}

\begin{proof} Remark \ref{rem:symbol} Part \eqref{comp rem} tells us that each of $V_1$, $V_2$ is complemented; using these complements, we may assume that we are working with trivial bundles.  For notational simplicity, however, we assume that $V_1$ and $V_2$ are the trivial \emph{line} bundle on $\sfG_0$; the straightforward generalization of the following proof to matrix valued versions of $C_e(A^*)$, $C^*_\cale(\calt \sfG)$, and $C^*_\cale(\sfG)$ leads to the result for general $V_1$, $V_2$; we leave the details to the reader.   Moreover, using Remark \ref{rem:symbol} Part \eqref{hom rem} we may assume that $\sigma$ is a $\complex$-valued function on $A^*$ that is homogeneous of degree zero outside some neighborhood of the zero section.   Fix now $(x,\lambda)\in \sfG_0\times (0,1]$.  When $V=V_1=V_2$ is the trivial line bundle, the Hilbert spaces $H^{x,\lambda}_1$ and $H^{x,\lambda}_2$ both identify with $L^2(\sfG^x, \mu^{x,\lambda})$.

According to its definition, $Q^{x,\lambda}(\sigma)$ has the Schwartz kernel $\mathfrak{K}(\sigma)(g, h)$ defined for $g,h\in \sfG^x$ by 
\[
\mathfrak{K}(\sigma)(g,h):=\frac{1}{(2\pi)^d}\int_{A^*_{s(g),\lambda}}e^\lambda_{-\xi}(hg^{-1}) \sigma(s(g), \xi)d\xi.
\]
From its expression, we can see that $Q^{x,\lambda}(\sigma)$ is a pseudodifferential operator on $H$ with symbol $\sigma(s(g), \xi)$. As $e^\lambda_\xi(-)$ is supported within the $R$-neighborhood of $x$ in $\sfG^x$ for the scaled metric, the kernel $\mathfrak{K}(\sigma)$ of $Q^{x,\lambda}(\sigma)$ is supported within the subset $\{(g,h)\in \sfG^x\times \sfG^x \mid d^{x,\lambda}(g,h)<R\}$ of $\sfG^x\times \sfG^x$.  

Using Assumption \ref{ass:kappa} and Lemma \ref{lem:cover}, there is an open cover $\mathcal{U}^x$ of $\sfG^x$ with the properties listed there with respect to the constants $r$ and $R$; note that $\mathcal{U}^x$ is necessarily countable.  Let $\{\varphi_i\}_{i\in \N}$ be a smooth partition of unity on $\sfG^x$ subordinate to the cover $\mathcal{U}^x$.   Let $u\in L^2(\sfG^x, \mu^{x,\lambda})$ and write $u=\sum_i \varphi_i u$, where $\varphi_iu\in L^2(\sfG^x, \mu^{x,\lambda})$. As $0\leq \varphi\leq 1$, 
\[
\|\varphi_i u\|^2_{L^2(\sfG^x, \mu^{x,\lambda})}\leq \int_{\sfG^x} \varphi_i |u|^2 \mu^{x,\lambda}. 
\]
Therefore, 
\[
\sum_i \|\varphi_i u\|^2\leq \sum_i \int_{\sfG^x} \varphi_i |u|^2\mu^{x,\lambda}\leq \|u\|^2.
\] 
The conditions in Lemma \ref{lem:cover} imply that given $U_i\in \calu^x$ there are at most $m$ other open sets $U_j$ in $\calu^x$ that intersects $U_i$ nontrivially. Accordingly, we compute
\[
\begin{split}
\Big\|\sum_{i\geq N} \varphi_i u\Big\|^2=\Big\langle \sum_{i\geq N} \varphi_i u, \sum_{j\geq N} \varphi_j u\Big\rangle &\leq \sum_{i,j\geq N}|\langle \varphi_i u, \varphi_j u\rangle | \leq \sum_{U_i \cap U_j \ne \varnothing, ~i,j\geq N} \frac{1}{2}(\|\varphi_i u\|^2+\|\varphi_j u\|^2)\\
&\leq \frac{m+1}{2} \sum_{i\geq N} \|\varphi_i u\|^2=\frac{m+1}{2}\sum_{i\geq N} \|\varphi_i u\|^2\\
&\leq \frac{m+1}{2}\int \sum_{i\geq N} \varphi_i |u|^2 \mu^{x,\lambda}.
\end{split}
\]
From the above estimate, we see that $\sum_{i\leq N}\varphi_i u\to u$ as $N\to \infty$ in $L^2(\sfG^x, \mu^{x,\lambda})$. 

Note that $Q^{x,\lambda}(\sigma)$ is an order zero pseudodifferential operator on $\sfG^x$, and therefore is a bounded linear operator on $L^2(\sfG^x, \mu^{x,\lambda})$ by essentially the same argument as used in the proof of \cite[Theorem 18.1.11]{ho:book3}. Hence, for $u\in L^2(\sfG^x, \mu^{x,\lambda})$, we can compute $Q^{x,\lambda}(\sigma)(u)$ by 
\[
Q^{x,\lambda}(\sigma)(u)=\sum_i Q^{x,\lambda}(\sigma)(\varphi_i u),
\]
and so we have that $\langle Q^{x,\lambda}(\sigma)(u), u\rangle$ equals
\[
\Big\langle \sum_i Q^{x,\lambda}(\sigma)(\varphi_i u), \sum_j \varphi_j u\Big\rangle=\sum_{i,j}\langle Q^{x,\lambda}(\sigma)(\varphi_i u), \varphi_j u\rangle. 
\]
The inner product  $\langle Q^{x,\lambda}(\sigma)\left(\varphi_i u\right), \varphi_j u\rangle$ is non-zero only when the supports of $Q^{x,\lambda}(\sigma)\left(\varphi_i u\right)$ and $\varphi_j u$ have non-trivial intersection.  We continue the above computation by 
\[
\begin{split}
|\sum_{i,j}\langle Q^{x,\lambda}(\sigma)(\varphi_i u), \varphi_j u\rangle|&\leq \sum_{i,j}|\langle Q^{x,\lambda}(\sigma)(\varphi_i u), \varphi_j u\rangle|\\
&=\sum_{\operatorname{supp}(Q^{x,\lambda}(\sigma)(\varphi_i u))\cap \operatorname{supp}(\varphi_ju)\ne \varnothing } |\langle Q^{x,\lambda}(\sigma)(\varphi_i u), \varphi_j u\rangle|\\
&\leq \sum_{\operatorname{supp}(Q^{x,\lambda}(\sigma)(\varphi_iu ))\cap \operatorname{supp}(\varphi_ju)\ne \varnothing } \frac{1}{2}(\|Q^{x,\lambda}(\sigma)(\varphi_i u)\|^2+\|\varphi_j u\|^2)
\end{split}
\]

We have assumed that $\nabla^k\tilde{\chi}$ is uniformly bounded on the unit sphere $\mathbb{S}(\otimes ^k T\sfG^x)$ of $\otimes^kT\sfG^x$ for $0\leq k\leq 2d+1$. Furthermore, we have assumed that $\nabla^k\widehat{\sigma}$ is uniformly bounded on $\mathbb{S}(TT^*\sfG^x)$ for $0\leq k\leq 2d+1$.  It is not hard to check from this that the function $\tilde{\sigma}(g, h, \xi):=\tilde{\chi}(hg^{-1})\widehat{\sigma}(s(g), \xi)$ satisfies the assumption of Lemma \ref{lem:sobolev}, once we restrict its support to any suitably small ball diffeomorphic to Euclidean space. The support of $\varphi_i u$ is contained inside some ball $U_i$ of radius $r$ with center $z_i$.  As the function $\tilde{\chi}(gh^{-1})$ is zero if $d(g,h)\geq R$,  the support of $Q^{x,\lambda}(\sigma)(\varphi_i u)$ is inside the geodesic ball centered at $z_i$ with radius at most $r+R<\iota$. We can then use Lemma \ref{lem:sobolev} (plus a change of variables to take into account the dependence of the metric and volume on $\lambda$) to conclude that 
\[
\|Q^{x,\lambda}(\sigma)(\varphi_i u)\|_{L^2(\sfG^x, \mu^{x,\lambda})}\leq D N^{k, k, 2d+1}(\tilde{\sigma})\|\varphi_i u\|_{L^2(\sfG^x, \mu^{x,\lambda})}\leq M\|\phi_iu\|_{L^2(\sfG^x,\mu^{x,\lambda})}, 
\]
where $M$ is an absolute constant, depending only on the size of $\nabla^k\widehat{\sigma}$ and $\nabla^k\tilde{\chi}$, and the $k$th-order derivatives of the Riemannian curvature tensor and dimension of $\sfG^x$ (and not on $\lambda$).
Hence we can bound $\|Q^{x,\lambda}(\sigma)(\varphi_i u)\|^2+\|\varphi_j u\|^2$ from above by 
\[
M^2 \|\varphi_i u\|^2+ \|\varphi_j u\|^2.
\]

By the assumption on the cover $\calu^x$, the support $\operatorname{supp}(Q^{x,\lambda}(\sigma)(\varphi_i u))$ intersects at most $m$ the supports $\operatorname{supp}(\varphi_ju)$. We can continue the above estimate of 
$$ \sum_{\operatorname{supp}(Q^{x,\lambda}(\sigma)(\varphi_i u))\cap \operatorname{supp}(\varphi_ju)\ne \varnothing } \frac{1}{2}(\|Q^{x,\lambda}(\sigma)(\varphi_i u)\|^2+\|\varphi_j u\|^2)$$ by the following inequalities
\[
\begin{split}
& \sum_{\operatorname{supp}(Q^{x,\lambda}(\sigma)(\varphi_i u))\cap \operatorname{supp}(\varphi_ju)\ne \varnothing } \frac{1}{2}(\|Q^{x,\lambda}(\sigma)(\varphi_i u)\|^2+\|\varphi_j u\|^2)\\
\leq& \sum_{i} \frac{1}{2} (M^2+m) \|\varphi_i u\|^2\leq  \frac{1}{2} (M^2+m) \sum_i \|\varphi_i u\|^2\\
\leq & \frac{1}{2} (M^2+m) \|u\|^2. 
\end{split}
\]

In summary, we have proved that for fixed $\lambda$, 
\[
|\langle Q^{x,\lambda}(\sigma)(u), u\rangle|\leq \frac{1}{2} (M^2+m) \|u\|^2,
\]
where $M$ and $m$ are constants independent of $\lambda$.  This implies
$$
\|Q^{x,\lambda}(\sigma)\|\leq 2(M^2+m),
$$ 
and we are done.
\end{proof}

Our next goal is to patch these operators $Q^{x,\lambda}$ together in a suitable sense to give a globally defined multiplier of $C^*_\cale(\calt\sfG)$.  For notational simplicity, we will keep to the case where $V_1$ and $V_2$ are trivial, leaving the (minor) extra details necessary in the matricial case to the reader.  

Define
$$
H:=\bigoplus_{x\in \sfG_0}L^2(A_x)\oplus \bigoplus_{(x,\lambda)\in \sfG_0\times (0,1]}L^2(\sfG^x,\mu^{x,\lambda})
$$
and let
$$
\rho:C_\cale(\calt \sfG) \to \mathcal{B}(H)
$$  
be the direct sum of the representations defining the norm on $C_\cale(\calt \sfG)$ (see Definition \ref{def:red roe}), so $\rho$ extends to a faithful representation of the Roe algebra $C^*_\cale(\calt \sfG)$.  Let $Q(\sigma)$ be the operator on $H$ which acts as $Q^{x,\lambda}(\sigma)$ on each summand $L^2(\sfG^x,\mu^{x,\lambda})$, and by fiberwise convolution by the fiberwise Fourier transform $\mathcal{F}(\sigma^x)$ of $\sigma$ on each $L^2(A_x)$.  Combining Proposition \ref{lem:bound} with basic estimates on the norms of the convolution operators, we see that $Q(\sigma)$ is a well-defined, bounded operator on $H$.

\begin{lem}\label{lem:mult}
We have following inclusions, 
\[
Q(\sigma)\cdot \rho(C_\cale(\calt\sfG))\subset \rho(C_\cale(\calt\sfG)),\qquad \rho(C_\cale(\calt\sfG))\cdot Q(\sigma)\subset \rho(C_\cale(\calt\sfG)). 
\]
\end{lem}

\begin{proof}
For simplicity, we just look at the product $Q(\sigma)\rho(f)$ for $f\in C_\cale(\calt\sfG)$; the other case is similar.   We now start computing.  

First consider, $\lambda>0$, and let 
\[
\mathfrak{K}^{x,\lambda}(\sigma)(g,h):=\frac{1}{(2\pi)^d}\int_{A^*_{s(g), \lambda}}e^\lambda_{-\xi}(hg^{-1}) \sigma(s(g), \xi)d\xi
\]
denote the Schwartz kernel of $Q^{x,\lambda}(\sigma)$; here the function $e^\lambda_{-\xi}(-)$ is defined in the same way as the function $e_{-\xi}(-)$ in Equation (\ref{eq:funct-e}), but using the $\lambda$-scaled metric on $A$.   Then the Schwartz kernel of the operator $Q(\sigma)\rho(f)$ acting on $L^2(\sfG^x,\mu^{x,\lambda})$ is given by
\begin{align*}
\mathfrak{L}^{x,\lambda}(\sigma)(g,h) & =\int_{\sfG^x}\mathfrak{K}^{x,\lambda}(\sigma)(g,k) f(kh^{-1})d\mu^{x,\lambda}(k) \\
&=\int_{\sfG^x}d\mu^{x, \lambda}(k) f(kh^{-1})\frac{1}{(2\pi)^d}\int_{A^*_{s(g),\lambda}}e^\lambda_{-\xi}(kg^{-1}) \sigma(s(g), \xi)d\xi. 
\end{align*}
Set $k'=kh^{-1}$, allowing us to rewrite the above integral as 
\[
\int_{\sfG^{s(h)}}d\mu^{s(h), \lambda}(k') f(k')\frac{1}{(2\pi)^d}\int_{A^*_{s(g),\lambda}}e^\lambda_{-\xi}(k'hg^{-1}) \sigma(s(g), \xi)d\xi 
\]
Define a function $\ell_\lambda:\sfG\to \C$ by 
\[
\ell_\lambda(g)=\int_{\sfG^{x}}d\mu^{x,\lambda}(k') f(k')\frac{1}{(2\pi)^d}\int_{A^*_{s(g),\lambda}}e^\lambda_{-\xi}(k'g^{-1}) \sigma(s(g), \xi)d\xi. 
\]

At this point, we know that the restriction of $Q(\sigma)\rho(f)$ to each $L^2(\sfG^x,\mu^{x,\lambda})$ has Schwartz kernel $(g,h)\mapsto \ell_\lambda(hg^{-1})$. We claim that $\ell_\lambda$ actually belongs to $C_\cale(\sfG)$, which follows from the next two observations. 
\begin{enumerate}
\item As  $f|_{\calt\sfG^{x,\lambda}}$ is supported in a compact set, the Fourier transform
\[
\int_{\sfG^x} d\mu^{x, \lambda}(k')e^\lambda_{-\xi}(k'g^{-1}) f(k')
\] 
is a function of $\xi$ decaying rapidly at infinity. Therefore, 
\[
\frac{1}{(2\pi)^d}\int_{A^*_{s(g),\lambda}}\sigma(s(g), \xi)d\xi  \int_{\sfG^x} d\mu^{x,\lambda}(k')e^\lambda_{-\xi}(k'g^{-1})f(k')
\]  
is a continuous (actually, smooth) function of $g$. 
\item  As $f(k')$ is supported within an $S>0$ tube of the unit space, when $g$ is outside the $S+R$ tube of the unit space, $e^\lambda_{-\xi}(k'g^{-1})f(k')=0$. Therefore $\ell_{\lambda}$ is supported within the $S+R$ tube of $\sfG_0$.  
\end{enumerate}

Define now $\ell:\calt\sfG\to \C$ by stipulating that the restriction $\ell_\lambda$ of $\ell$ to the fiber over $\lambda$ is given by
\[
\left\{\begin{array}{ll}
\ell_\lambda(g):=\int_{\sfG^{x}}d\mu^{x,\lambda }(k') f_\lambda (k')\frac{1}{(2\pi)^d}\int_{A^*_{s(g),\lambda}}e^\lambda_{-\xi}(k'g^{-1}) \sigma(s(g), \xi)d\xi&\lambda>0\\
\ell_\lambda(x,v):=\calf\big(\sigma(x)\big)\ast f(x) (v)&\lambda=0,
\end{array}
\right.
\]
where $\calf\big(\sigma(x)\big)$ is the fiberwise Fourier transform of $\sigma(x, -)$.  From our work so far, if $\ell$ can be shown to be continuous, then it is in $C_\cale(\calt\sfG)$ and we will have $Q(\sigma)\rho(f)=\rho(\ell)$.  The continuity of $\ell|_{\sfG_0\times (0,1]}$ and $\ell_{\sfG_0\times \{0\}}$ follow directly from the formulas and our work above, so we are left to check the compatibility of these two cases. 

Let $\{g_\lambda\}_{\lambda\in (0,1]}$ be a family in $\calt\sfG$ such that $g_\lambda \to (x, v)\in A$ as $\lambda\to 0$, where $A$ is identified with the fiber of $\calt\sfG$ over $0$. More explicitly, this means that as $\lambda\to 0$, $g_\lambda\to x$, and $\exp^{-1}_{t(g_\lambda)}(g_\lambda)\to v$ (it is important here that the exponential maps are defined with respect to the $\lambda$-scaled metrics).  We want to show that $\ell_\lambda (g_\lambda)$ converges to $\ell_0(x,v)$.  Indeed, 
\[
\ell_\lambda(g_\lambda)=\int_{t(g_\lambda)} d\mu^{t(g_\lambda),\lambda }(k')f_\lambda(k')\frac{1}{(2\pi)^d} \int_{A^*_{s(g_\lambda),\lambda}} e^\lambda_{-\xi}(k'g_\lambda^{-1})\sigma(s(g_\lambda), \xi)d\xi
\]
Write $k'=\exp_{t(g_\lambda)}(w)$, and let $J$ be the Jacobian of this exponential exponential map. We can rewrite $\ell_\lambda(g_\lambda)$ as
\[
\ell_\lambda(g_\lambda)=\int_{A_{t(g_\lambda),\lambda}} J_\lambda dw f(\exp_{t(g_\lambda)} (w))\frac{1}{(2\pi)^d} \int_{A^*_{s(g_\lambda),\lambda}} e^\lambda_{-\xi}(\exp_{t(g_\lambda)}(w) g_\lambda^{-1})\sigma(s(g_\lambda), \xi)d\xi
\]
As $f$ is in $C_\cale(\calt\sfG)$, $f( \exp_{t(g_\lambda)} (w))\to f(x, w)$ as $\lambda\to 0$. As $\lambda\to 0$, $J_\lambda\to 1$. And as $\lambda\to 0$, 
\[
e^\lambda_{-\xi}(\exp_{t(g_\lambda)}(w) g_\lambda^{-1})\to \exp(2\pi \sqrt{-1}\langle w-v, \xi\rangle). 
\]
From the above estimate, we can conclude $\ell_\lambda(g_\lambda)\to \ell_0(x, v)$, $\lambda\to 0$, completing the proof. 
\end{proof}

\begin{thm} \label{thm:k-theory} 
 Assume the Euclidean structure on $A$ satisfies Assumption \ref{ass:kappa}. Given a smooth symbol $(V_1, V_2, \sigma)$ representing an element in $K_0(C^*_e(A^*))$, the operator $Q(\sigma)$ defines a $K$-theory element $[Q(\sigma)]$ in $K_0(C^*_\cale (\calt \sfG))$ such that $ev_{0*}([Q(\sigma)])=[V_1, V_2, \sigma]$. 
\end{thm}

\begin{proof} Following the proofs of Proposition \ref{lem:bound} and Lemma \ref{lem:mult}, we assume for notational simplicity that $V_1$ and $V_2$ are the trivial line bundle on $\sfG_0$ and $\sigma$ is a $\complex$-valued function on $A^*$ that is homogeneous of degree zero and invertible outside of some neighborhood of the zero section.    

Lemma \ref{lem:mult} combined with boundedness of $Q(\sigma)$ shows that $Q(\sigma)$ identifies canonically with an element of the multiplier algebra of $C^*_{\mathcal{E}}(\calt \sfG)$ (see for example \cite[Section 3.12]{ped}). 
Now, say $\tau$ is the partial inverse to $\sigma$ that appears in the definition of a smooth symbol.  Consider the operators $Q(\sigma)\circ Q(\tau)-I$ and $Q(\tau)\circ Q(\sigma)-I$.  Using the symbolic calculus of pseudodifferential operator theory, c.f. \cite[Theorem 18.1.23]{ho:book3}, \cite{Pf-Po-Ta:lie-groupoid-index}, given any $x\in \sfG_0$ and $\lambda\in (0,1]$, $Q^{(x,\lambda)}(\sigma)\circ Q^{(x,\lambda)}(\tau)-I$ is a pseudodifferential operator on $\sfG^x$ of order 0 with principal symbol the appropriate restriction of $\sigma\tau-1$. By Definition \ref{dfn:geometry-K}, $\sigma\tau-1$ is the Fourier transform of a function supported in some tube around the zero section of uniform width.  From this, together with similar (and simpler) arguments when $\lambda=0$, we can derive that both  $Q(\sigma)\circ Q(\tau)-I$ and $Q(\tau)\circ Q(\sigma)-I$ belong to $C^*_\cale(\calt\sfG)$. 

At this point, we have that $Q(\sigma)$ is an element of the multiplier algebra of $C^*_\cale(\calt \sfG)$, which is invertible modulo the ideal $C^*_\cale(\calt \sfG)$.  It thus defines an element of $K_0(C^*_\cale(\calt \sfG))$: precisely, if $[Q(\sigma)]_1$ is the element of the $K_1$-group of the multiplier algebra of $C^*_\cale(\calt\sfG)$, then the element $[Q(\sigma)]\in K_0(C^*_\cale(\calt\sfG))$ we get is the image of $[Q(\sigma)]_1$ under the boundary map.  The identity that $ev_{0, *}[W(\sigma)]=[V_1, V_2, \sigma]$ follows easily from the definition of $Q(\sigma)$ and of the latter class. 
\end{proof}

Observe from Theorem \ref{thm:k-theory} that at $\lambda=1$, $Q(\sigma)$ restricts to $Q^{x,1}(\sigma)$ on each $t$-fiber $\sfG^x$.  Accordingly, $ev_{1*}[Q(\sigma)]=[Q^{1}(\sigma)]\in K_0(C^*_\cale(\sfG))$, where $[Q^1(\sigma)]$ is the $K$-theory class defined by the invertible multiplier $Q^1(\sigma)$ of $C^*_\cale(\sfG)$, having identified $\sfG$ with the fiber of $\calt\sfG$ for $\lambda=1$.  We conclude this section with the following corollary, which says that the analytic index of a smooth symbol $[V_1, V_2, \sigma]$ is given by the (higher, or $K$-theoretic) index of the associated pseudodifferential operator in the classical sense.
\begin{cor}\label{cor:index} Under Assumption \ref{ass:kappa}, the analytic index $\Ind_a([V_1, V_2, \sigma])$ of the class $[V_1, V_2, \sigma]$ in $K_0(C_e(A^*))$ of a smooth symbol  is computed by $[Q^1(\sigma)]\in K_0(C^*_\cale(\sfG))$.  \qed
\end{cor}

\begin{rem}\label{rmk:ex-bounded-geom} When the unit space $\sfG_0$ is closed, i.e. compact without boundary, Assumption  \ref{ass:kappa} holds automatically.  Furthermore, every element in $K_0(C_e(A^*))$ has a representative that is a smooth symbol as in Definition \ref{dfn:geometry-K}. Therefore, Corollary \ref{cor:index} applies to equate the analytic index of a $K$-theory class $\sigma$ of $A^*$ with the index of its quantization $Q(\sigma)$. In this section, we have extended this result to groupoids with noncompact unit space under Assumption \ref{ass:kappa} on bounded geometry.  
\end{rem}

\section{Scalar curvature obstructions}\label{sc sec}
In this section, we apply the theory developed above to study (regular) foliations with leafwise positive scalar metrics.  We will discuss some specific instances of regular foliations; nonetheless, we develop the theory in a more general setting that this as the extra generality does not cause any extra difficulties.  

We assume throughout that $\sfG$ is an $s$-connected Lie groupoid with a Lie algebroid $A$, and that $A$ is equipped with a Euclidean structure $g_A$ that satisfies Assumption \ref{ass:kappa}.  We first consider two natural examples with this structure that arise from regular foliations on closed manifolds.

\subsection{Two examples}\label{sec:example}
Let $M$ be a closed (compact without boundary) manifold with a regular foliation $\calf$.  
In this subsection, we discuss two examples of groupoids naturally associated to $\calf$ with Riemannian metrics on the Lie algebroids satisfying the above properties.  
\begin{enumerate}
\item Let $\pi:\widehat{M}\to M$ be the universal covering space. As $\pi$ is a local diffeomorphism, $\calf$ lifts to a regular foliation $\widehat{\calf}$ on $\widehat{M}$.   Let $\mathcal{H}_{\widehat{\calf}}$ and $\mathcal{H}_{\calf}$ be the holonomy groupoids on $\widehat{M}$ and $M$ defined by holonomy equivalent paths along leaves of $\widehat{\calf}$ and $\calf$.  Let $g$ be a Riemannian metric on $M$. Accordingly, $g$ induces Euclidean metrics $g_\calf$ and $g_{\widehat{\calf}}$ on $\calf$ and $\widehat{\calf}$.  

By Proposition \ref{prop:example-entourage}, $\mathcal{H}_{\widehat{\calf}}$ is equipped with a right-invariant system of metrics $\{d^x\}_{x\in \widehat{M}}$. As $M$ is compact, on every leaf of $\calf$ and also of $\widehat{\calf}$,  directional derivatives of the Riemannian curvature tensor of any order are uniformly bounded and injectivity radii are uniformly bounded away from zero.  For every $x\in \widehat{M}$, by the right invariance of the metrics on $\mathcal{H}_{\widehat{\calf}}$, $\mathcal{H}_{\widehat{\calf}}^x$ is an isometric covering space of the leaf $\mathcal{L}_x$ of $\widehat{\calf}$ containing $x$. As leaves of $\widehat{\calf}$ have Riemannian curvature tensors with uniformly bounded directional derivatives and injectivity radii uniformly bounded below, $(\mathcal{H}_{\widehat{\calf}}, \{d_x\})_{x\in \widehat{M}}$ satisfies Assumption \ref{ass:kappa}.

\item Let $N_\calf=TM/\calf$ be the normal bundle of $\calf$ in $TM$.  Define a fiber bundle $\pi:\mathcal{M}\to M$ by stipulating that the fiber $\pi^{-1}(x)$ over $x\in M$ is the space of euclidean metrics on the normal space $N_\calf|_x$. 

The normal bundle $N_\calf$ is equipped with the Bott connection, which is a flat connection along the direction of the foliation $\mathcal{F}$. The Bott connection induces a partial Ehresmann connection on the fiber bundle $\calm$ along the direction of $\calf$. As the Bott connection is flat along $\calf$, the Ehresmann connection on the fiber bundle $\calm$ defines a foliation $\widetilde{\calf}$ on $\calm$, where every leaf of the foliation $\widetilde{\calf}$ is a covering space of the corresponding leaf of $\calf$ on $M$ with $\pi$ the covering map.  

Choose a Riemannian metric $g$ on $M$, which induces Riemannian metrics $g_\calf$ and $g_{\widetilde{\calf}}$ on $\calf$ and $\widetilde{\calf}$. The leaves of $\calf$ have uniformly bounded directional derivative (of any order) Riemannian curvature tensor and injectivity radii uniformly bounded from below. Since the restriction of $\pi$ to each leaf of $\widetilde{\calf}$ is a covering map, the leaves of $\widetilde{\calf}$ also have uniformly bounded directional derivative Riemannian curvature tensor and injectivity radii uniformly  bounded from below. Let $\mathcal{H}_{\widetilde{\calf}}$ be the holonomy groupoid of the regular foliation $\widetilde{\calf}$ on $\calm$. For every $x\in \calm$, each $t$-fiber $\calh_{\widetilde{\calf}}^x$ of $\calh_{\widetilde{\calf}}$ is an isometric covering space of the corresponding leaf $\call_x$ on $\calm$. Therefore, the $t$-fibers of $\calh_{\widetilde{\calf}}$ have Riemannian curvature tensors with uniformly bounded directional derivatives  and injectivity radii uniformly bounded below. 
\end{enumerate}

In either case, as the leaves of $\calf$ are all complete, the leaves of $\widehat{\calf}$ and $\widetilde{\calf}$, which are covering spaces of the leaves of $\calf$, are also complete. Therefore the $t$-fibers $\calh^x_{\widehat{\calf}}$ and $\calh_{\widetilde{\calf}}^x$ are all complete and proper metric spaces. With the above properties of leafwise bounded geometry, Assumption \ref{ass:kappa} holds true on the groupoids $\calh_{\widehat{\calf}}$ and $\calh_{\widetilde{\calf}}$. Therefore, we can use Definition \ref{def:red roe} to define the Roe $C^*$-algebras $C^*_{\mathcal{E}}(\calh_{\widetilde{F}})$ and $C^*_{\mathcal{E}}(\calh_{\widehat{F}})$ and study the analytic indices of longitudinal elliptic differential operators.

\subsection{Leafwise positive scalar metric}
In this subsection, we study the analytic index of a spin Dirac operator $\slashed{D}$. 

Let $g_A$ be a metric on the Lie algebroid $A$. The right translation of the groupoid $\sfG$ on itself is proper and free. The right translation of the metric $g_A$ defines an $\sfG$-invariant Riemannian metric on each $t$-fiber $\sfG^x$ for $x\in X$.  The Levi-Civita connection $\nabla$ on each $\sfG^x$ is invariant under the $\sfG$-action. The associated curvature of $\nabla$ is also $\sfG$-invariant. Therefore, the scalar curvature $k$ is a smooth function on $\sfG$ and $\sfG$-invariant. Therefore the scalar curvature $k$ is the pullback of a smooth function on $X$. 

We now restrict attention to a Lie groupoid $\sfG$ whose $t$-fibers $\{\sfG^x\}$ are even dimensional and equipped with $\sfG$-invariant spin structure.  Moreover, we assume that there is a constant $\epsilon>0$ satisfying  $k_{g_A}(x)>\epsilon$ for every $x\in \sfG_0$.  Let $\cals^{\pm}_x$ be the even and odd parts of the spinor bundles on $\sfG^x$ for $x\in X$ (the spinor bundle splits into even and odd parts by our even-dimensionality assumption).  On each $\sfG^x$, consider the leafwise spin Dirac operator $\slashed{D}^+_x:\Gamma(\cals^+_x)\to \Gamma(\cals^-_x)$ and its (formal) adjoint $\slashed{D}^-_x:\Gamma(\cals^-_x)\to \Gamma(\cals^+_x)$.  The Lichnerowicz formula combined with our assumption on scalar curvature implies that 
$$
\slashed{D}^+_x\slashed{D}^-_x\geq \frac{\epsilon}{4} \quad \text{and} \quad \slashed{D}^-_x\slashed{D}^+_x\geq \frac{\epsilon}{4} 
$$
for all $x\in X$.   It follows that $(\slashed{D}^+_x\slashed{D}^-_x)^{-1/2}$ makes sense, and that $P_x^+:=\slashed{D}^+_x(\slashed{D}^-_x\slashed{D}^+_x)^{-1/2}$ is an invertible order zero operator from the Hilbert space completion $H^x_+$ of $\cals^x_+$ to the Hilbert space completion $H^x_-$ of $\cals^x_-$.  

Note that all this structure is $\sfG$-invariant, and therefore in particular $\cals_+$ and $\cals_-$ are pullbacks of bundles on $X$.  Abusing notation, write $\sigma(\slashed{D}^+_x)$ for the symbol of $P_x^+$.  It is straightforward to check using the pseudodifferential calculus, the standard description of the symbol of the Dirac operator and Assumption \ref{ass:kappa} on uniformly bounded directional derivatives of the Riemannian curvature tensor that $(\cals_+,\cals_-,\sigma(\slashed{D}_+))$ is a smooth symbol in the sense of Definition \ref{dfn:geometry-K}, and thus defines a $K$-theory class for $K_0(C^*_e(A^*))$.  Using Assumption \ref{ass:kappa} allows us to apply Theorem \ref{thm:k-theory} and Corollary \ref{cor:index}. We conclude that the analytic index $\Ind_a([\cals^+, \cals^-, \sigma(\slashed{D}_+)])$ is equal to the index defined by $Q^1(\sigma(\slashed{D}_+))$ in $K_0(C^*_\cale(\sfG))$.  Note that we have that $[Q^1(\sigma(\slashed{D}))]=[P_+]$ in $K_0(C^*_\cale(\sfG))$, and as $P_+$ is invertible, $[P_+]=0$. Thus we can conclude that 
\[
\Ind_a\Big([\cals^+, \cals^-, \sigma(\slashed{D}_+)]\Big)=0, 
\]

To summarize, we have the following theorem.
\begin{thm}\label{thm:vanishing}
Let $\sfG$ be a Lie groupoid, and assume that the associated Lie algebroid $A$ has even rank, and is equipped with a spin structure such that the associated Riemannian metric $g_A$ satisfies Assumption \ref{ass:kappa}.  Then the triple $(\cals^+, \cals^-, \sigma(\slashed{D}_+))$ associated to the Dirac operator is a smooth symbol in the sense of Definition \ref{dfn:geometry-K}.  Moreover, if 
$g_A$ has uniformly positive scalar curvature, the analytic index of the K-theory class $[\cals^+, \cals^-, \sigma(\slashed{D}_+)]$ vanishes in $K_0(C^*_\cale(\sfG))$, i.e. 
\[
\Ind_a\Big([\cals^+, \cals^-, \sigma(\slashed{D}_+)]\Big)=0. \qedhere
\]
\end{thm}

Consider now the two groupoids $\calh_{\widehat{\calf}}$ and $\calh_{\widetilde{\calf}}$ associated to a regular foliation $\calf$ on a closed manifold $M$ as in Section \ref{sec:example}. If we assume that $\calf$ is equipped with a leafwise spin structure and has positive scalar curvature at every point of $M$,  the Lie algebroids of the groupoids $\calh_{\widehat{\calf}}$ and  $\calh_{\widetilde{\calf}}$ in Section \ref{sec:example} are equipped with spin structures such that the associated Riemannian metrics have uniformly positive scalar curvatures. As a corollary of Theorem \ref{thm:vanishing}, we have the following vanishing result about the indices of the leafwise Dirac operators. 
\begin{cor}\label{cor:vanishing}Assume that the regular foliation $\calf$ on a closed manifold $M$ has even rank and a spin structure with uniformly positive scalar curvature $k\geq \epsilon>0$. Let $\slashed{D}_+$ be the leafwise Dirac operator on $M$, so $\slashed{D}_+$ lifts to the corresponding leafwise spin Dirac operators $\widehat{\slashed{D}}_+$ and $\widetilde{\slashed{D}}_+$ on $\widehat{M}$ and $\widetilde{M}$. Let $\sigma(\widehat{\slashed{D}}_+))$ and $\sigma(\widetilde{\slashed{D}}_+)$ be the associated symbol classes in $K_0(C_e(\widehat{\calf}^*))$ and $K_0(C_e(\widetilde{\calf}^*))$ as defined above. The analytic indices of these symbols vanish, i.e. 
\begin{eqnarray*}
\Ind_a\Big( \sigma(\widehat{\slashed{D}}_+)\Big)=0,\qquad 
\Ind_a\Big( \sigma(\widetilde{\slashed{D}}_+)\Big)=0
\end{eqnarray*}
(where we abuse notation by omitting the bundles involved).
\end{cor}
\begin{proof}
As was explained before, if $\calf$ is equipped with a spin structure with positive scalar curvature $k\geq \epsilon>0$, Assumptions \ref{ass:kappa} holds on $\calh_{\widehat{\calf}}$ and $\calh_{\widetilde{\calf}}$.  We conclude from the above discussion that all the assumptions of Theorem \ref{thm:vanishing} hold for both $\widehat{\slashed{D}}_+$ and $\widetilde{\slashed{D}}_+$, so the desired vanishing property follows from Theorem \ref{thm:vanishing}. 
\end{proof}

\section{Appendix}
In this Appendix, we prove the following lemma. Such a result is known to experts, but we could not find a suitable reference so include a proof here for the reader's convenience. 

\begin{lem}\label{lem:cover} Let $\iota$, $\kappa$, $r$, and $R$ be positive constants, and $d$ a positive integer.  Then for any $R>0$, there exists $m\in \mathbb{N}$ depending only on $\iota, \kappa, r, R$, and $d$ with the following property. Let $M$ be a $d$-dimensional complete Riemannian manifold with all sectional curvatures in the interval $[-\kappa, \kappa]$ and injectivity radius bounded below by $\iota$. There exists a cover $\mathcal{U}$ of $M$ by balls of radius $r$ such that for any $U\in \mathcal{U}$ there are at most $m$ balls $V\in\mathcal{U}$ with 
 \[
 V\cap N_R(U)\ne \varnothing,
 \]
 where $N_R(U):=\{y\in M | d(y, U)<R\}$.  Moreover, if $r<\iota$, then each $U\in \mathcal{U}$ is diffeomorphic to $\R^d$.
\end{lem}

\begin{proof} Using Zorn's lemma, we can choose a maximal subset $Z$ of $M$ such that $d(x,y)\geq r$ for all $x,y\in Z$. Let 
\[
\calu:=\{B(z; r) \mid z\in Z\},
\]
where $B(z; r)$ is the geodesic ball center at $z$ with radius $r$.  If $r<\iota$, each $B(z; r)$ is diffeomorphic to an open ball in $\mathbb{R}^d$ via the exponential map.  Note that as $Z$ is maximal, the union of $B(z; r)$ in $\calu$ must cover the whole $M$, and so $\calu$ is an open cover of $M$. 

Fix $U=B(z; r)\in \calu$, and let $S_z$ be the set of $y\in Z$ such that 
\[
B(y;r)\cap B(z; R+r)\ne \varnothing.
\]
We consider the collection of $B(y; r/2)$ for $y\in S_z$.  The triangle inequality shows that for all $y\in S_z$,
\[
B(z; R+3r)\supseteq B(y; r)\supseteq B(y; r/2).
\]
Furthermore, as the points in $S_z$ are at least $r$-distance apart, the sets $\{B(y;r/2)\mid y\in S_z\}$ are mutually disjoint. Hence we have the following inclusion
\[
B(z; R+3r)\supseteq \bigsqcup_{z_i\in S_z} B(z_i; r/2).
\] 

By the Bishop-Gromov theorem \cite[Theorem 107, Page 310]{be03}, there is a constant $C>0$ depending only only $\kappa$, $d$, and $R+3r$ such that 
\[
\operatorname{Volume}(B(z; R+3r))\leq C. 
\]
On the other hand, a standard result in comparison theory \cite[Theorem 103, Page 306]{be03} gives us a constant $c>0$ depending only on $r$, $\kappa$, $\iota$ and $d$ such that 
\[
\operatorname{Volume}(B(y; r/2))\geq \operatorname{Volume}(B(y;\min(r/2,\iota))\geq c,
\]
for each $y\in S_z$. 

Combing these two volume bounds, we see that 
\[
C\geq \operatorname{Volume}(B(z; R+3r)) \geq \operatorname{Volume}\Big(\bigsqcup_{y\in S_z}B(y; r/2)\Big)\geq |S_z| c. 
\]
We conclude that $|S_z|\leq C/c$, and note that $C/c$ only depends on $\iota$, $\kappa$, $R$, $r$, and $d$. We choose $m$ to be an integer greater than or qual to $C/c$, and therefore complete the proof. 
\end{proof}

\vspace{2mm}

{\small \noindent{Xiang Tang}, Department of Mathematics, Washington
University, St. Louis, MO, 63130, U.S.A.,
Email: xtang@math.wustl.edu.

\vspace{2mm}

\noindent{Rufus Willett}, Department of Mathematics, University of Hawai'i, Honolulu, HI, 96822, U.S.A.,
Email: rufus.willett@hawaii.edu.

\vspace{2mm}
\noindent{Yi-Jun Yao}, School of Mathematical Sciences, Fudan
University, Shanghai 200433, P.R.China., Email:
yaoyijun@fudan.edu.cn.

}
\end{document}